\documentclass[11pt]{preprint}
\usepackage{difftrees} %Added%
\usepackage[full]{textcomp}
\usepackage[osf]{newtxtext} 
\usepackage[cal=boondoxo]{mathalfa}
\usepackage{colortbl}

\usepackage{tikz-cd} %Added%
\usepackage{float}

\usetikzlibrary{cd} %Added%
\usepackage[all,cmtip]{xy}
\usepackage{comment}

\usepackage{amssymb}
\usepackage{mathtools}
\usepackage{hyperref}
\usepackage{breakurl}
\usepackage{mhenvs}
\usepackage{mhequ} 
\usepackage{mhsymb}
\usepackage{booktabs}
\usepackage{tikz}
\usepackage{tcolorbox}
\usepackage{mathrsfs}
\usepackage[utf8]{inputenc}
\usepackage{longtable}
\usepackage{wrapfig}
\usepackage{rotating} %Added%
\usepackage{subcaption}
\usepackage{mathrsfs}
\usepackage{epsfig}
\usepackage{microtype}
\usepackage{comment}
\usepackage{wasysym}
\usepackage{centernot}
\usepackage{enumitem}
\usepackage{bm}
\usepackage{stackrel}

\usepackage{graphicx}
\usepackage{axodraw}
\usepackage{xspace}
\usepackage{subcaption} %Added%
\usepackage{epsfig} %Added%
\usepackage{axodraw} %Added%
\usepackage{xspace} %Added%
\usepackage[toc,page]{appendix} %Added%
\usepackage[all,cmtip]{xy} %Added%
\usepackage{relsize} %Added%
\usepackage{shuffle} %Added%
\makeatletter
\newcommand{\globalcolor}[1]{%
  \color{#1}\global\let\default@color\current@color
}
\makeatother

\usetikzlibrary{calc}
\usetikzlibrary{decorations}
\usetikzlibrary{positioning}
\usetikzlibrary{shapes}
\usetikzlibrary{external}

\definecolor{blush}{rgb}{0.87, 0.36, 0.51}
	\definecolor{brightcerulean}{rgb}{0.11, 0.67, 0.84}
	\definecolor{greenryb}{rgb}{0.4, 0.69, 0.2}

\newif\ifdark
\darkfalse

\ifdark
\definecolor{darkred}{rgb}{0.9,0.2,0.2}
\definecolor{darkblue}{rgb}{0.7,0.3,1}
\definecolor{darkgreen}{rgb}{0.1,0.9,0.1}
\definecolor{franck}{rgb}{0,0.8,1}
\definecolor{pagebackground}{rgb}{.15,.21,.18}
\definecolor{pageforeground}{rgb}{.84,.84,.85}
\pagecolor{pagebackground}
\AtBeginDocument{\globalcolor{pageforeground}}
\definecolor{symbols}{rgb}{0,0.7,1}
\colorlet{connection}{red!80!black}
\colorlet{boxcolor}{blue!50}

\else

\definecolor{darkred}{rgb}{0.7,0.1,0.1}
\definecolor{darkblue}{rgb}{0.4,0.1,0.8}
\definecolor{darkgreen}{rgb}{0.1,0.7,0.1}
\definecolor{franck}{rgb}{0,0,1}
\definecolor{pagebackground}{rgb}{1,1,1}
\definecolor{pageforeground}{rgb}{0,0,0}
\colorlet{symbols}{blue!90!black}
\colorlet{connection}{red!30!black}
\colorlet{boxcolor}{blue!50!black}

\fi

\def\slash{\leavevmode\unskip\kern0.18em/\penalty\exhyphenpenalty\kern0.18em}
\def\dash{\leavevmode\unskip\kern0.18em--\penalty\exhyphenpenalty\kern0.18em}

\DeclareMathAlphabet{\mathbbm}{U}{bbm}{m}{n}

\DeclareFontFamily{U}{BOONDOX-calo}{\skewchar\font=45 }
\DeclareFontShape{U}{BOONDOX-calo}{m}{n}{
  <-> s*[1.05] BOONDOX-r-calo}{}
\DeclareFontShape{U}{BOONDOX-calo}{b}{n}{
  <-> s*[1.05] BOONDOX-b-calo}{}
\DeclareMathAlphabet{\mcb}{U}{BOONDOX-calo}{m}{n}
\SetMathAlphabet{\mcb}{bold}{U}{BOONDOX-calo}{b}{n}
\setlist{noitemsep,topsep=4pt,leftmargin=1.5em}

\DeclareMathAlphabet{\mathbbm}{U}{bbm}{m}{n}

\DeclareMathAlphabet{\mcb}{U}{BOONDOX-calo}{m}{n}
\SetMathAlphabet{\mcb}{bold}{U}{BOONDOX-calo}{b}{n}
\DeclareFontFamily{U}{mathx}{\hyphenchar\font45}
\DeclareFontShape{U}{mathx}{m}{n}{
      <5> <6> <7> <8> <9> <10>
      <10.95> <12> <14.4> <17.28> <20.74> <24.88>
      mathx10
      }{}
\DeclareSymbolFont{mathx}{U}{mathx}{m}{n}
\DeclareMathSymbol{\bigtimes}{1}{mathx}{"91}

\def\emptyset{{\centernot\ocircle}}

\providecommand{\figures}{false}
{ \ifthenelse{\equal{\figures}{false}} {#1}{\[ {\rm Figure \ missing !} \]} }{}

\newcommand{\bn}{{\mathbf n}}

\newcommand{\cA}{{\mathcal A}}

\def\CA{\mathcal{A}}

\def\CM{\mathcal{M}}

\tikzstyle{tinydots}=[dash pattern=on \pgflinewidth off \pgflinewidth]
\tikzstyle{superdense}=[dash pattern=on 4pt off 1pt]

\newcommand{\mcR}{\mathcal{R}}

\newcommand{\mcL}{\mathcal{L}}

\newcommand{\length}[1]{\mathcal{l}({#1})}

\newcommand{\coord}{{\mcR}}

\newcommand{\beq}{\begin{equation}}
\newcommand{\eeq}{\end{equation}}

\usepackage{empheq}

\newcommand{\mfL}{\mathfrak{L}}

\newcommand{\mfl}{\mathfrak{l}}

\def\${|\!|\!|}

\newcounter{theorem}

\newtheorem{ex}[theorem]{Example}

\newenvironment{DIFnomarkup}{}{} % see man latexdiff

\theorembodyfont{\rmfamily}
\newtheorem{example}[lemma]{Example}

\newcommand{\rrightarrow}{{\to\hskip -4.9mm\raise 1pt\hbox{$\to$}}}

\newfont{\indic}{bbmss12}

\def\Nabla_#1{\nabla_{\!#1}}

\usepackage{forest}
\forestset{
	decor/.style = {
		label/.expanded = {[inner sep = 0.2ex, font=\unexpanded{\tiny}]right:{$#1$}}
	},
	root/.style = {minimum size = 0.1ex},
	decorated/.style = {
		for tree = {
			circle, fill, inner sep = 0.3ex, minimum size = 0.4ex,
			edge=thick,
			grow' = north, l = 0, l sep = 1.0ex, s sep = 0.7em,
			fit = tight, parent anchor = center, child anchor = center,
			delay = {decor/.option = content, content =}
		}
	},
	default preamble = {decorated, root},
}

\makeatletter
\pgfdeclareshape{crosscircle}
{
  \inheritsavedanchors[from=circle] % this is nearly a circle
  \inheritanchorborder[from=circle]
  \inheritanchor[from=circle]{north}
  \inheritanchor[from=circle]{north west}
  \inheritanchor[from=circle]{north east}
  \inheritanchor[from=circle]{center}
  \inheritanchor[from=circle]{west}
  \inheritanchor[from=circle]{east}
  \inheritanchor[from=circle]{mid}
  \inheritanchor[from=circle]{mid west}
  \inheritanchor[from=circle]{mid east}
  \inheritanchor[from=circle]{base}
  \inheritanchor[from=circle]{base west}
  \inheritanchor[from=circle]{base east}
  \inheritanchor[from=circle]{south}
  \inheritanchor[from=circle]{south west}
  \inheritanchor[from=circle]{south east}
  \inheritbackgroundpath[from=circle]
  \foregroundpath{
    \centerpoint%
    \pgf@xc=\pgf@x%
    \pgf@yc=\pgf@y%
    \pgfutil@tempdima=\radius%
    \pgfmathsetlength{\pgf@xb}{\pgfkeysvalueof{/pgf/outer xsep}}%  
    \pgfmathsetlength{\pgf@yb}{\pgfkeysvalueof{/pgf/outer ysep}}%  
    \ifdim\pgf@xb<\pgf@yb%
      \advance\pgfutil@tempdima by-\pgf@yb%
    \else%
      \advance\pgfutil@tempdima by-\pgf@xb%
    \fi%
    \pgfpathmoveto{\pgfpointadd{\pgfqpoint{\pgf@xc}{\pgf@yc}}{\pgfqpoint{-0.707107\pgfutil@tempdima}{-0.707107\pgfutil@tempdima}}}
    \pgfpathlineto{\pgfpointadd{\pgfqpoint{\pgf@xc}{\pgf@yc}}{\pgfqpoint{0.707107\pgfutil@tempdima}{0.707107\pgfutil@tempdima}}}
    \pgfpathmoveto{\pgfpointadd{\pgfqpoint{\pgf@xc}{\pgf@yc}}{\pgfqpoint{-0.707107\pgfutil@tempdima}{0.707107\pgfutil@tempdima}}}
    \pgfpathlineto{\pgfpointadd{\pgfqpoint{\pgf@xc}{\pgf@yc}}{\pgfqpoint{0.707107\pgfutil@tempdima}{-0.707107\pgfutil@tempdima}}}
  }
}
\makeatother

\def\symbol#1{\textcolor{symbols}{#1}}

\def\decorate#1#2{
        \ifnum#2>0
    		\foreach \count in {1,...,#2}{
	       	let
				\p1 = (sourcenode.center),
                \p2 = (sourcenode.east),
				\n1 = {\x2-\x1},
				\n2 = {1mm},
				\n3 = {(1.3+0.6*(\count-1))*\n1},
				\n4 = {0.7*\n1}
			in 
        		node[rectangle,fill=symbols,rotate=30,inner sep=0pt,minimum width=0.2*\n2,minimum height=\n2] at ($(sourcenode.center) + (\n3,\n4)$) {}
				}
		\fi
        \ifnum#1>0
    		\foreach \count in {1,...,#1}{
	       	let
				\p1 = (sourcenode.center),
                \p2 = (sourcenode.east),
				\n1 = {\x2-\x1},
				\n2 = {1mm},
				\n3 = {(1.3+0.6*(\count-1))*\n1},
				\n4 = {0.7*\n1}
			in 
        		node[rectangle,fill=symbols,rotate=-30,inner sep=0pt,minimum width=0.2*\n2,minimum height=\n2] at ($(sourcenode.center) + (-\n3,\n4)$) {}
				}
		\fi
}

\tikzset{
    dectriangle/.style 2 args={
        triangle,
        alias=sourcenode,
        append after command={\decorate{#1}{#2}}
    },
    dectriangle/.default={0}{0},
}

\tikzset{
	cross/.style={path picture={ 
  		\draw[symbols]
			(path picture bounding box.south east) -- (path picture bounding box.north west) (path picture bounding box.south west) -- (path picture bounding box.north east);
		}},
root/.style={circle,fill=green!50!black,inner sep=0pt, minimum size=1.2mm},
        dot/.style={circle,fill=pageforeground,inner sep=0pt, minimum size=1mm},
        dotred/.style={circle,fill=pageforeground!50!pagebackground,inner sep=0pt, minimum size=2mm},
        var/.style={circle,fill=pageforeground!10!pagebackground,draw=pageforeground,inner sep=0pt, minimum size=3mm},
        kernel/.style={semithick,shorten >=2pt,shorten <=2pt},
        kernels/.style={snake=zigzag,shorten >=2pt,shorten <=2pt,segment amplitude=1pt,segment length=4pt,line before snake=2pt,line after snake=5pt,},
        rho/.style={densely dashed,semithick,shorten >=2pt,shorten <=2pt},
           testfcn/.style={dotted,semithick,shorten >=2pt,shorten <=2pt},
        renorm/.style={shape=circle,fill=pagebackground,inner sep=1pt},
        labl/.style={shape=rectangle,fill=pagebackground,inner sep=1pt},
        xic/.style={very thin,circle,draw=symbols,fill=symbols,inner sep=0pt,minimum size=1.2mm},
        g/.style={very thin,rectangle,draw=symbols,fill=symbols!10!pagebackground,inner sep=0pt,minimum width=2.5mm,minimum height=1.2mm},
        xi/.style={very thin,circle,draw=symbols,fill=symbols!10!pagebackground,inner sep=0pt,minimum size=1.2mm},
	xies/.style={very thin,rectangle,fill=green!50!black!25,draw=symbols,inner sep=0pt,minimum size=1.1mm},
	xiesf/.style={very thin,rectangle,fill=green!50!black,draw=symbols,inner sep=0pt,minimum size=1.1mm},
        xix/.style={very thin,crosscircle,fill=symbols!10!pagebackground,draw=symbols,inner sep=0pt,minimum size=1.2mm},
        X/.style={very thin,cross,rectangle,fill=pagebackground,draw=symbols,inner sep=0pt,minimum size=1.2mm},
	xib/.style={thin,circle,fill=symbols!10!pagebackground,draw=symbols,inner sep=0pt,minimum size=1.6mm},
	xie/.style={thin,circle,fill=green!50!black,draw=symbols,inner sep=0pt,minimum size=1.6mm},
	xid/.style={thin,circle,fill=symbols,draw=symbols,inner sep=0pt,minimum size=1.6mm},
	xibx/.style={thin,crosscircle,fill=symbols!10!pagebackground,draw=symbols,inner sep=0pt,minimum size=1.6mm},
	kernels2/.style={very thick,draw=connection,segment length=12pt},
	keps/.style={thin,draw=symbols,->},
	kepspr/.style={thick,draw=connection,->},
	krho/.style={thin,draw=symbols,superdense,->},
	krhopr/.style={thick,draw=connection,superdense},
	triangle/.style = { regular polygon, regular polygon sides=3},
	not/.style={thin,circle,draw=connection,fill=connection,inner sep=0pt,minimum size=0.5mm},
	diff/.style = {very thin,draw=symbols,triangle,fill=red!50!black,inner sep=0pt,minimum size=1.6mm},
	diff1/.style = {very thin,dectriangle={1}{0},fill=red!50!black,draw=symbols,inner sep=0pt,minimum size=1.6mm},
	diff2/.style = {very thin,dectriangle={1}{1},fill=red!50!black,draw=symbols,inner sep=0pt,minimum size=1.6mm},
		diffmini/.style = {very thin,rectangle,fill=black,draw=black,inner sep=0pt,minimum size=0.75mm},
	 kernelsmod/.style={very thick,draw=connection,segment length=12pt},
	 rec/.style = {very thin,rectangle,fill=black,draw=black,inner sep=0pt,minimum size=2mm},
	cerc/.style={very thin,circle,draw=black,fill=symbols,inner sep=0pt,minimum size=2mm},
	stars/.style={very thin,star,star points=6,star point ratio=0.5, draw=black,fill=red,inner sep=0pt,minimum size=0.7mm},
	>=stealth,
        }
        \tikzset{
root/.style={circle,fill=black!50,inner sep=0pt, minimum size=3mm},
        circ/.style={circle,fill=white,draw=black,very thin,inner sep=.5pt, minimum size=1.2mm},
        round1/.style={fill=white,outer sep = 0,inner sep=2pt,rounded corners=1mm,draw,text=black,thin,minimum size=1.2mm},
          circ1/.style={circle,fill=red!10,draw=red,very thin,inner sep=.5pt, minimum size=1.2mm},
        rect/.style={fill=white,outer sep = 0,inner sep=2pt,rectangle,draw,text=black,thin,minimum size=1.2mm},
        rect1/.style={fill=white,outer sep = 0,inner sep=2pt,rectangle,draw,text=black,thin,minimum size=1.2mm},
        round2/.style={fill=red!10,outer sep = 0,inner sep=2pt,rounded corners=1mm,draw,text=black,thin,minimum size=1.2mm},
       round3/.style={fill=blue!10,outer sep = 0,inner sep=2pt,rounded corners=1mm,draw,text=black,thin,minimum size=1.2mm}, 
        rect2/.style={fill=black!10,outer sep = 0,inner sep=2pt,rectangle,draw,text=black,thin,minimum size=1.2mm},
        dot/.style={circle,fill=black,inner sep=0pt, minimum size=1.2mm},
        dotred/.style={circle,fill=black!50,inner sep=0pt, minimum size=2mm},
        var/.style={circle,fill=black!10,draw=black,inner sep=0pt, minimum size=3mm},
        kernel/.style={semithick,shorten >=2pt,shorten <=2pt},
         diag/.style={thin,shorten >=4pt,shorten <=4pt},
        kernel1/.style={thick},
        kernels/.style={snake=zigzag,shorten >=2pt,shorten <=2pt,segment amplitude=1pt,segment length=4pt,line before snake=2pt,line after snake=5pt,},
		kernels1/.style={snake=zigzag,segment amplitude=0.5pt,segment length=2pt},
		rho1/.style={densely dotted,semithick},
        rho/.style={densely dashed,semithick,shorten >=2pt,shorten <=2pt},
           testfcn/.style={dotted,semithick,shorten >=2pt,shorten <=2pt},
           visible/.style={draw, circle, fill, inner sep=0.25ex},
        renorm/.style={shape=circle,fill=white,inner sep=1pt},
        labl/.style={shape=rectangle,fill=white,inner sep=1pt},
        xic/.style={very thin,circle,fill=symbols,draw=black,inner sep=0pt,minimum size=1.2mm},
        xi/.style={very thin,circle,fill=blue!10,draw=black,inner sep=0pt,minimum size=1.2mm},
	xib/.style={very thin,circle,fill=blue!10,draw=black,inner sep=0pt,minimum size=1.6mm},
	xie/.style={very thin,circle,fill=green!50!black,draw=black,inner sep=0pt,minimum size=1mm},
	xid/.style={very thin,circle,fill=symbols,draw=black,inner sep=0pt,minimum size=1.6mm},
	edgetype/.style={very thin,circle,draw=black,inner sep=0pt,minimum size=5mm},
	nodetype/.style={very thick,circle,draw=black,inner sep=0pt,minimum size=5mm},
	kernels2/.style={very thick,draw=connection,segment length=12pt},
clean/.style={thin,circle,fill=black,inner sep=0pt,minimum size=1mm},	not/.style={thin,circle,fill=symbols,draw=connection,fill=connection,inner sep=0pt,minimum size=0.8mm},
	>=stealth,
        }

\makeatletter
\def\DeclareSymbol#1#2#3{%
	\expandafter\gdef\csname MH@symb@#1\endcsname{\tikzsetnextfilename{symbol#1}%
	\tikz[baseline=#2,scale=0.15,draw=symbols,line join=round]{#3}}%
	\expandafter\gdef\csname MH@symb@#1s\endcsname{\scalebox{0.75}{\tikzsetnextfilename{symbol#1}%
	\tikz[baseline=#2,scale=0.15,draw=symbols,line join=round]{#3}}}%
	\expandafter\gdef\csname MH@symb@#1ss\endcsname{\scalebox{0.65}{\tikzsetnextfilename{symbol#1}%
	\tikz[baseline=#2,scale=0.15,draw=symbols,line join=round]{#3}}}%
	}
\def\<#1>{\ifthenelse{\boolean{mmode}}{\mathchoice{\csname MH@symb@#1\endcsname}{\csname MH@symb@#1\endcsname}{\csname MH@symb@#1s\endcsname}{\csname MH@symb@#1ss\endcsname}}{\csname MH@symb@#1\endcsname}}
\makeatother

\DeclareSymbol{Xi22}{0.5}{\draw (0,0) node[xi] {} -- (-1,1) node[not] {} -- (0,2) node[xi] {};} % 1 not used in the text
\DeclareSymbol{Xi2}{-2}{\draw (-1,-0.25) node[xi] {} -- (0,1) node[xi] {};} % 2
\DeclareSymbol{Xi2b}{-2}{\draw (-1,-0.25) node[xic] {} -- (0,1) node[xic] {};} % 2
\DeclareSymbol{Xi2g}{-2}{\draw (-1,-0.25) node[xies] {} -- (0,1) node[xi] {};} % 2
\DeclareSymbol{Xi2g2}{-2}{\draw (-1,-0.25) node[xi] {} -- (0,1) node[xies] {};} % 2
\DeclareSymbol{cXi2}{-2}{\draw (0,-0.25) node[xi] {} -- (-1,1) node[xic] {};}%3
\DeclareSymbol{Xi3}{0}{\draw (0,0) node[xi] {} -- (-1,1) node[xi] {} -- (0,2) node[xi] {};}%4
\DeclareSymbol{XiIIXi}{0}{\draw (0,0) node[xi] {} -- (-1,1); \draw[kernels2] (-1,1) node[not] {} -- (0,2) node[xi] {};}%4

\DeclareSymbol{Xi4}{2}{\draw (0,0) node[xi] {} -- (-1,1) node[xi] {} -- (0,2) node[xi] {} -- (-1,3) node[xi] {};}%5
\DeclareSymbol{Xi4_1}{2}{\draw (0,0) node[xic] {} -- (-1,1) node[xic] {} -- (0,2) node[xi] {} -- (-1,3) node[xi] {};}%6
\DeclareSymbol{Xi4_2}{2}{\draw (0,0) node[xic] {} -- (-1,1) node[xi] {} -- (0,2) node[xi] {} -- (-1,3) node[xic] {};}%7
\DeclareSymbol{Xi2X}{-2}{\draw (0,-0.25) node[xi] {} -- (-1,1) node[xix] {};}%8
\DeclareSymbol{XXi2}{-2}{\draw (0,-0.25) node[xix] {} -- (-1,1) node[xi] {};}%9
\DeclareSymbol{IIXi}{0}{\draw (0,-0.25) node[not] {} -- (-1,1) node[xi] {} -- (0,2) node[xi] {};}%10
\DeclareSymbol{IXi^2}{-1}{\draw (-1,1) node[xi] {} -- (0,0) node[not] {} -- (1,1) node[xi] {};}%11
\DeclareSymbol{IIXi^2}{-4}{\draw (0,-1.5) node[not] {} -- (0,0);
\draw[kernels2] (-1,1) node[xi] {} -- (0,0) node[not] {} -- (1,1) node[xi] {};}%12
\DeclareSymbol{XiX}{-2.8}{\node[xibx] {};}%13
\DeclareSymbol{tauX}{-2.8}{ \node[X] {};}%14
\DeclareSymbol{Xi}{-2.8}{\node[xib] {};}%15

\DeclareSymbol{IXiX}{-1}{\draw (0,-0.25) node[not] {} -- (-1,1) node[xix] {};}%18
\DeclareSymbol{IXi3}{2}{\draw (0,-0.25) node[not] {} -- (-1,1) node[xi] {} -- (0,2) node[xi] {} -- (-1,3) node[xi] {};}%20
\DeclareSymbol{IXi}{-2}{\draw (0,-0.25) node[not] {} -- (-1,1) node[xi] {};}%21
\DeclareSymbol{XiI}{-2}{\draw (0,-0.25) node[xi] {} -- (-1,1) node[not] {};}%22

\DeclareSymbol{Xi4b}{0}{\draw(0,1.5) node[xi] {} -- (0,0); \draw (-1,1) node[xi] {} -- (0,0) node[xi] {} -- (1,1) node[xi] {};}%23
\DeclareSymbol{Xi4b'}{0}{\draw(0,1.5) node[xi] {} -- (0,-0.2); \draw (-1,1) node[xi] {} -- (0,-0.2) node[not] {} -- (1,1) node[xi] {};}%24 not used
\DeclareSymbol{Xi4c}{0}{\draw (0,1) -- (0.8,2.2) node[xi] {};\draw (0,-0.25) node[xi] {} -- (0,1) node[xi] {} -- (-0.8,2.2) node[xi] {};}%25
\DeclareSymbol{Xi4d}{-4.5}{\draw (0,-1.5) node[not] {} -- (0,0); \draw (-1,1) node[xi] {} -- (0,0) node[xi] {} -- (1,1) node[xi] {};}%26 not used
\DeclareSymbol{Xi4e}{0}{\draw (0,2) node[xi] {} -- (-1,1) node[xi] {} -- (0,0) node[xi] {} -- (1,1) node[xi] {};}%27
\DeclareSymbol{Xi4e'}{0}{\draw (0,2) node[xi] {} -- (-1,1) node[xi] {} -- (0,-0.2) node[not] {} -- (1,1) node[xi] {};}%28 not used

\DeclareSymbol{Xitwo}%29
{0}{\draw[kernels2] (0,0) node[not] {} -- (-1,1) node[not] {}
-- (-2,2) node[not]{} -- (-3,3) node[xi]  {};
\draw[kernels2] (0,0) -- (1,1) node[xi] {};
\draw[kernels2] (-1,1) -- (0,2) node[xi] {};
\draw[kernels2] (-2,2) -- (-1,3) node[xi] {};}

\DeclareSymbol{IXitwo}%30
{0}{\draw (-.7,1.2) node[xi] {} -- (0,-0.2) -- (.7,1.2) node[xi] {};}
\DeclareSymbol{I1Xitwo}%30
{0}{\draw[kernels2] (0,0) node[not] {} -- (-1,1) node[xi] {};
\draw[kernels2] (0,0) -- (1,1) node[xi] {};}

\DeclareSymbol{I1Xitwobis}%30
{0}{\draw[kernels2] (0,0) node[not] {} -- (-1,1) node[xies] {};
\draw[kernels2] (0,0) -- (1,1) node[xies] {};}

\DeclareSymbol{I1Xitwog}%30
{0}{\draw[kernels2] (0,0) node[not] {} -- (-1,1) node[xies] {};
\draw[kernels2] (0,0) -- (1,1) node[xi] {};}

\DeclareSymbol{cI1Xitwo}%31
{0}{\draw[kernels2] (0,0) node[not] {} -- (-1,1) node[xic] {};
\draw[kernels2] (0,0) -- (1,1) node[xi] {};}

\DeclareSymbol{I1IXi3}{0}{\draw (0,0) node[xi] {} -- (-1,1) ; %32
\draw[kernels2] (-1,1) node[not] {} -- (0,2) node[xi] {};
\draw[kernels2] (-1,1) node[not] {} -- (-2,2) node[xi] {};}

\DeclareSymbol{I1Xi3c}{-1}{\draw[kernels2](0,1.5) node[xi] {} -- (0,0) node[not] {}; \draw (-1,1) node[xi] {} -- (0,0) ; \draw[kernels2] (0,0) -- (1,1) node[xi] {};}%33

\DeclareSymbol{I1Xi3cbis}{-1}{\draw[kernels2](0,1.5) node[xies] {} -- (0,0) node[not] {}; \draw (-1,1) node[xies] {} -- (0,0) ; \draw[kernels2] (0,0) -- (1,1) node[xies] {};}%33

\DeclareSymbol{I1IXi3b}{0}{\draw[kernels2] (0,0) node[not] {} -- (-1,1) ; \draw[kernels2] (0,0)   -- (1,1) node[xi] {} ;
\draw (-1,1) node[xi] {} -- (0,2) node[xi] {};
}%34

\DeclareSymbol{I1IXi3c}{0}{\draw[kernels2] (0,0) node[not] {} -- (-1,1) ; \draw[kernels2] (0,0)   -- (1,1) node[xi] {} ;
\draw[kernels2] (-1,1) node[not] {} -- (0,2) node[xi] {};
\draw[kernels2] (-1,1) node[not] {} -- (-2,2) node[xi] {};}%35

\DeclareSymbol{I1IXi3cbis}{0}{\draw[kernels2] (0,0) node[not] {} -- (-1,1) ; \draw[kernels2] (0,0)   -- (1,1) node[xies] {} ;
\draw[kernels2] (-1,1) node[not] {} -- (0,2) node[xies] {};
\draw[kernels2] (-1,1) node[not] {} -- (-2,2) node[xies] {};}%35

\DeclareSymbol{I1Xi}{0}{\draw[kernels2] (0,0) node[not] {} -- (-1,1)  node[xi] {} ;}%36

\DeclareSymbol{I1Xi4a}{2}{\draw[kernels2] (0,0) node[not] {} -- (-1,1) ; \draw[kernels2] (0,0) node[not] {} -- (1,1) node[xi] {} ;%37
\draw (-1,1) node[xi] {} -- (0,2) node[xi] {} -- (-1,3) node[xi] {};}
\DeclareSymbol{cI1Xi4a}{2}{\draw[kernels2] (0,0) node[not] {} -- (-1,1) ; \draw[kernels2] (0,0) node[not] {} -- (1,1) node[xic] {} ;%39
\draw (-1,1) node[xic] {} -- (0,2) node[xi] {} -- (-1,3) node[xi] {};}
\DeclareSymbol{I1Xi4b}{2}{\draw (0,0) node[xi] {} -- (-1,1) node[xi] {} -- (0,2) ; \draw[kernels2] (0,2) node[not] {} -- (-1,3) node[xi] {};\draw[kernels2] (0,2)  -- (1,3) node[xi] {};
}%41

\DeclareSymbol{cI1Xi4b}{2}{\draw (0,0) node[xic] {} -- (-1,1) node[xic] {} -- (0,2) ; \draw[kernels2] (0,2) node[not] {} -- (-1,3) node[xi] {};\draw[kernels2] (0,2)  -- (1,3) node[xi] {};
}%42

\DeclareSymbol{I1Xi4c}{2}{\draw (0,0) node[xi] {} -- (-1,1) node[not] {}; \draw[kernels2] (-1,1) -- (0,2) ; 
\draw[kernels2] (-1,1) -- (-2,2) node[xi] {} ;
\draw (0,2) node[xi] {} -- (-1,3) node[xi] {};}%43

\DeclareSymbol{cI1Xi4c}{2}{\draw (0,0) node[xic] {} -- (-1,1) node[not] {}; \draw[kernels2] (-1,1) -- (0,2) ; 
\draw[kernels2] (-1,1) -- (-2,2) node[xic] {} ;
\draw (0,2) node[xi] {} -- (-1,3) node[xi] {};}%44

\DeclareSymbol{I1Xi4ab}{2}{\draw[kernels2] (0,0) node[not] {} -- (-1,1) ; \draw[kernels2] (0,0) node[not] {} -- (1,1) node[xi] {};\draw (-1,1) node[xi] {} -- (0,2) ; \draw[kernels2] (0,2) node[not] {} -- (-1,3) node[xi] {};\draw[kernels2] (0,2)  -- (1,3) node[xi] {}; }%45

\DeclareSymbol{cI1Xi4ab}{2}{\draw[kernels2] (0,0) node[not] {} -- (-1,1) ; \draw[kernels2] (0,0) node[not] {} -- (1,1) node[xic] {};\draw (-1,1) node[xic] {} -- (0,2) ; \draw[kernels2] (0,2) node[not] {} -- (-1,3) node[xi] {};\draw[kernels2] (0,2)  -- (1,3) node[xi] {}; }%46

\DeclareSymbol{I1Xi4bc}{2}{\draw (0,0) node[xi] {} -- (-1,1) node[not] {}; \draw[kernels2] (-1,1) -- (0,2) ; %47
\draw[kernels2] (-1,1) -- (-2,2) node[xi] {} ; \draw[kernels2] (0,2) node[not] {} -- (-1,3) node[xi] {};\draw[kernels2] (0,2)  -- (1,3) node[xi] {};
}

\DeclareSymbol{cI1Xi4bc}{2}{\draw (0,0) node[xic] {} -- (-1,1) node[not] {}; \draw[kernels2] (-1,1) -- (0,2) ; %48
\draw[kernels2] (-1,1) -- (-2,2) node[xic] {} ; \draw[kernels2] (0,2) node[not] {} -- (-1,3) node[xi] {};\draw[kernels2] (0,2)  -- (1,3) node[xi] {};
}

\DeclareSymbol{I1Xi4abcc1}{2}{\draw[kernels2] (0,0) node[not] {} -- (-1,1) node[not] {}%50
-- (-2,2) node[not]{} -- (-3,3) node[xic]  {};
\draw[kernels2] (0,0) -- (1,1) node[xic] {};
\draw[kernels2] (-1,1) -- (0,2) node[xi] {};
\draw[kernels2] (-2,2) -- (-1,3) node[xi] {};
}

\DeclareSymbol{I1Xi4abcc1b}{2}{\draw[kernels2] (0,0) node[not] {} -- (-1,1) node[not] {}%50
-- (-2,2) node[not]{} -- (-3,3) node[xi]  {};
\draw[kernels2] (0,0) -- (1,1) node[xic] {};
\draw[kernels2] (-1,1) -- (0,2) node[xic] {};
\draw[kernels2] (-2,2) -- (-1,3) node[xi] {};
}

\DeclareSymbol{I1Xi4abcc2}{2}{\draw[kernels2] (0,0) node[not] {} -- (-1,1) node[not] {}%51
-- (-2,2) node[not]{} -- (-3,3) node[xic]  {};
\draw[kernels2] (0,0) -- (1,1) node[xi] {};
\draw[kernels2] (-1,1) -- (0,2) node[xi] {};
\draw[kernels2] (-2,2) -- (-1,3) node[xic] {};
}

\DeclareSymbol{I1Xi4ac}{2}{\draw[kernels2] (0,0) node[not] {} -- (-1,1) ; \draw[kernels2] (0,0) node[not] {} -- (1,1) node[xi] {}; 
\draw[kernels2] (-1,1) node[not] {} -- (0,2) ; %52
\draw[kernels2] (-1,1) -- (-2,2) node[xi] {} ;
\draw (0,2) node[xi] {} -- (-1,3) node[xi] {};}

\DeclareSymbol{cI1Xi4ac}{2}{\draw[kernels2] (0,0) node[not] {} -- (-1,1) ; \draw[kernels2] (0,0) node[not] {} -- (1,1) node[xic] {}; 
\draw[kernels2] (-1,1) node[not] {} -- (0,2) ; %53
\draw[kernels2] (-1,1) -- (-2,2) node[xic] {} ;
\draw (0,2) node[xi] {} -- (-1,3) node[xi] {};}

\DeclareSymbol{I1Xi4acc1}{2}{\draw[kernels2] (0,0) node[not] {} -- (-1,1) ; \draw[kernels2] (0,0) node[not] {} -- (1,1) node[xic] {}; 
\draw[kernels2] (-1,1) node[not] {} -- (0,2) ; %54
\draw[kernels2] (-1,1) -- (-2,2) node[xi] {} ;
\draw (0,2) node[xic] {} -- (-1,3) node[xi] {};}

\DeclareSymbol{I1Xi4acc2}{2}{\draw[kernels2] (0,0) node[not] {} -- (-1,1) ; \draw[kernels2] (0,0) node[not] {} -- (1,1) node[xic] {}; 
\draw[kernels2] (-1,1) node[not] {} -- (0,2) ; %55
\draw[kernels2] (-1,1) -- (-2,2) node[xi] {} ;
\draw (0,2) node[xi] {} -- (-1,3) node[xic] {};}

\DeclareSymbol{2I1Xi4}{2}{\draw[kernels2] (0,0) node[not] {} -- (-1,1) node[not] {};
\draw[kernels2] (0,0) -- (1,1) node[not] {};%56
\draw[kernels2] (-1,1) -- (-1.5,2.5) node[xi] {};
\draw[kernels2] (-1,1) -- (-0.5,2.5) node[xi] {};
\draw[kernels2] (1,1) -- (0.5,2.5) node[xi] {};
\draw[kernels2] (1,1) -- (1.5,2.5) node[xi] {};
}

\DeclareSymbol{2I1Xi4dis}{2}{\draw[kernels2] (0,0) node[not] {} -- (-1,1) node[not] {};
\draw[kernels2] (0,0) -- (1,1) node[not] {};%56
\draw[kernels2] (-1,1) -- (-1.5,2.5) node[xies] {};
\draw[kernels2] (-1,1) -- (-0.5,2.5) node[xies] {};
\draw[kernels2] (1,1) -- (0.5,2.5) node[xies] {};
\draw[kernels2] (1,1) -- (1.5,2.5) node[xies] {};
}

\DeclareSymbol{2I1Xi4c1}{2}{\draw[kernels2] (0,0) node[not] {} -- (-1,1) node[not] {};%57
\draw[kernels2] (0,0) -- (1,1) node[not] {};
\draw[kernels2] (-1,1) -- (-1.5,2.5) node[xic] {};
\draw[kernels2] (-1,1) -- (-0.5,2.5) node[xi] {};
\draw[kernels2] (1,1) -- (0.5,2.5) node[xic] {};
\draw[kernels2] (1,1) -- (1.5,2.5) node[xi] {};
}

\DeclareSymbol{2I1Xi4c2}{2}{\draw[kernels2] (0,0) node[not] {} -- (-1,1) node[not] {};%58
\draw[kernels2] (0,0) -- (1,1) node[not] {};
\draw[kernels2] (-1,1) -- (-1.5,2.5) node[xic] {};
\draw[kernels2] (-1,1) -- (-0.5,2.5) node[xic] {};
\draw[kernels2] (1,1) -- (0.5,2.5) node[xi] {};
\draw[kernels2] (1,1) -- (1.5,2.5) node[xi] {};
}

\DeclareSymbol{2I1Xi4b}{2}{\draw[kernels2] (0,0) node[not] {} -- (-1,1) ;%59
\draw[kernels2] (0,0) -- (1,1);
\draw (-1,1) node[xi] {} -- (-1,2.5) node[xi] {};
\draw (1,1)  node[xi] {} -- (1,2.5) node[xi] {};
}

\DeclareSymbol{2I1Xi4bb}{2}{\draw[kernels2] (0,0) node[not] {} -- (-1,1) ;%59
\draw[kernels2] (0,0) -- (1,1);
\draw (-1,1) node[xi] {} -- (-1,2.5) node[xiesf] {};
\draw (1,1)  node[xi] {} -- (1,2.5) node[xic] {};
}

\DeclareSymbol{2I1Xi4c}{2}{\draw[kernels2] (0,0) node[not] {} -- (-1,1);%60
\draw[kernels2] (0,0) -- (1,1) node[not] {};
\draw (-1,1)  node[xi] {} -- (-1,2.5) node[xi] {};
\draw[kernels2] (1,1) -- (0.4,2.5) node[xi] {};
\draw[kernels2] (1,1) -- (1.6,2.5) node[xi] {};
}

\DeclareSymbol{2I1Xi4cc1}{2}{\draw[kernels2] (0,0) node[not] {} -- (-1,1);%61
\draw[kernels2] (0,0) -- (1,1) node[not] {};
\draw (-1,1)  node[xic] {} -- (-1,2.5) node[xi] {};
\draw[kernels2] (1,1) -- (0.4,2.5) node[xic] {};
\draw[kernels2] (1,1) -- (1.6,2.5) node[xi] {};
}

\DeclareSymbol{2I1Xi4cc2}{2}{\draw[kernels2] (0,0) node[not] {} -- (-1,1);%62
\draw[kernels2] (0,0) -- (1,1) node[not] {};
\draw (-1,1)  node[xic] {} -- (-1,2.5) node[xic] {};
\draw[kernels2] (1,1) -- (0.4,2.5) node[xi] {};
\draw[kernels2] (1,1) -- (1.6,2.5) node[xi] {};
}

\DeclareSymbol{Xi4ba}{0}{\draw(-0.5,1.5) node[xi] {} -- (0,0); \draw (-1.5,1) node[xi] {} -- (0,0) node[not] {}; \draw[kernels2] (0,0) -- (1.5,1) node[xi] {};
\draw[kernels2] (0,0) -- (0.5,1.5) node[xi] {} ;}%63

\DeclareSymbol{Xi4badis}{0}{\draw(-0.5,1.5) node[xies] {} -- (0,0); \draw (-1.5,1) node[xies] {} -- (0,0) node[not] {}; \draw[kernels2] (0,0) -- (1.5,1) node[xies] {};
\draw[kernels2] (0,0) -- (0.5,1.5) node[xies] {} ;}%63

\DeclareSymbol{Xi4ba1}{0}{\draw(-0.5,1.5) node[xi] {} -- (0,0); \draw (-1.5,1) node[xi] {} -- (0,0) node[not] {}; \draw[kernels2] (0,0) -- (1.5,1) node[xic] {};
\draw[kernels2] (0,0) -- (0.5,1.5) node[xic] {} ;}%64

\DeclareSymbol{Xi4ba1b}{0}{\draw(-0.5,1.5) node[xic] {} -- (0,0); \draw (-1.5,1) node[xic] {} -- (0,0) node[not] {}; \draw[kernels2] (0,0) -- (1.5,1) node[xi] {};
\draw[kernels2] (0,0) -- (0.5,1.5) node[xi] {} ;}%64

\DeclareSymbol{Xi4ba1bdiff}{0}{\draw(-0.5,1.5) node[xic] {} -- (0,0); \draw (-1.5,1) node[xic] {} -- (0,0) node[not] {}; \draw (0,0) -- (1.5,1) node[xi] {};
\draw (0,0) -- (0.5,1.5) node[xi] {};
\draw(0,0) node[diff] {};}%64

\DeclareSymbol{Xi4ba1bb}{0}{\draw(-0.5,1.5) node[xic] {} -- (0,0); \draw (-1.5,1) node[xiesf] {} -- (0,0) node[not] {}; \draw[kernels2] (0,0) -- (1.5,1) node[xi] {};
\draw[kernels2] (0,0) -- (0.5,1.5) node[xi] {} ;}%64

\DeclareSymbol{Xi4ba2}{0}{\draw(-0.5,1.5) node[xi] {} -- (0,0); \draw (-1.5,1) node[xic] {} -- (0,0) node[not] {}; \draw[kernels2] (0,0) -- (1.5,1) node[xi] {};
\draw[kernels2] (0,0) -- (0.5,1.5) node[xic] {} ;}%65

\DeclareSymbol{Xi4ba2b}{0}{\draw(-0.5,1.5) node[xi] {} -- (0,0); \draw (-1.5,1) node[xic] {} -- (0,0) node[not] {}; \draw[kernels2] (0,0) -- (1.5,1) node[xi] {};
\draw[kernels2] (0,0) -- (0.5,1.5) node[xiesf] {} ;}%65

\DeclareSymbol{Xi4ca}{0}{\draw (0,1) -- (-1,2.2) node[xi] {};\draw (0,-0.25) node[xi] {} -- (0,1) ; \draw[kernels2] (0,1) node[not] {} -- (1,2.2) node[xi] {};%67
\draw[kernels2] (0,1) {} -- (0,2.7) node[xi] {};
}

\DeclareSymbol{Xi4cb}{0}{\draw (-1,1) -- (-2,2) node[xi] {};\draw[kernels2] (0,0)  -- (-1,1) node[xi] {} ; \draw[kernels2] (0,0) node[not] {} -- (1,1) node[xi] {} ; 
\draw (-1,1) node[xi] {} -- (0,2) node[xi] {};}

\DeclareSymbol{Xi4cbb}{0}{\draw (-1,1) -- (-2,2) node[xiesf] {};\draw[kernels2] (0,0)  -- (-1,1) node[xi] {} ; \draw[kernels2] (0,0) node[not] {} -- (1,1) node[xi] {} ; 
\draw (-1,1) node[xi] {} -- (0,2) node[xic] {};}

\DeclareSymbol{Xi4cbc1}{0}{\draw (-1,1) -- (-2,2) node[xic] {};\draw[kernels2] (0,0)  -- (-1,1) node[xic] {} ; \draw[kernels2] (0,0) node[not] {} -- (1,1) node[xi] {} ; 
\draw (-1,1) node[xic] {} -- (0,2) node[xi] {};}

\DeclareSymbol{Xi4cbc2}{0}{\draw (-1,1) -- (-2,2) node[xi] {};\draw[kernels2] (0,0)  -- (-1,1) node[xi] {} ; \draw[kernels2] (0,0) node[not] {} -- (1,1) node[xic] {} ; 
\draw (-1,1) node[xic] {} -- (0,2) node[xi] {};}

\DeclareSymbol{Xi4cab}{0}{\draw (-1,1) -- (-2,2) node[xi] {};\draw[kernels2] (0,0)  -- (-1,1); \draw[kernels2] (0,0) node[not] {} -- (1,1) node[xi] {} ; 
\draw[kernels2] (-1,1)  {} -- (0,2) node[xi] {};
\draw[kernels2] (-1,1) node[not] {} -- (-1,2.5) node[xi] {};
}%69

\DeclareSymbol{Xi4cabdis}{0}{\draw (-1,1) -- (-2,2) node[xies] {};\draw[kernels2] (0,0)  -- (-1,1); \draw[kernels2] (0,0) node[not] {} -- (1,1) node[xies] {} ; 
\draw[kernels2] (-1,1)  {} -- (0,2) node[xies] {};
\draw[kernels2] (-1,1) node[not] {} -- (-1,2.5) node[xies] {};
}%69

\DeclareSymbol{Xi4cabc1}{0}{\draw (-1,1) -- (-2,2) node[xi] {};\draw[kernels2] (0,0)  -- (-1,1); \draw[kernels2] (0,0) node[not] {} -- (1,1) node[xic] {} ; 
\draw[kernels2] (-1,1)  {} -- (0,2) node[xic] {};
\draw[kernels2] (-1,1) node[not] {} -- (-1,2.5) node[xi] {};
}%69

\DeclareSymbol{Xi4cabc2}{0}{\draw (-1,1) -- (-2,2) node[xic] {};\draw[kernels2] (0,0)  -- (-1,1); \draw[kernels2] (0,0) node[not] {} -- (1,1) node[xic] {} ; 
\draw[kernels2] (-1,1)  {} -- (0,2) node[xi] {};
\draw[kernels2] (-1,1) node[not] {} -- (-1,2.5) node[xi] {};
}%69

\DeclareSymbol{Xi4ea}{1.5}{\draw (-1,2.5) node[xi] {} -- (-1,1) node[xi] {} -- (0,0); %71
 \draw[kernels2] (0,0)  -- (1,1) node[xi] {};
\draw[kernels2] (0,0) node[not] {} -- (0,1.5) node[xi] {}; }

\DeclareSymbol{Xi4eac1}{1.5}{\draw (-1,2.5) node[xic] {} -- (-1,1) node[xi] {} -- (0,0); %72
 \draw[kernels2] (0,0)  -- (1,1) node[xic] {};
\draw[kernels2] (0,0) node[not] {} -- (0,1.5) node[xi] {}; }

\DeclareSymbol{Xi4eac1b}{1.5}{\draw (-1,2.5) node[xic] {} -- (-1,1) node[xi] {} -- (0,0); %72
 \draw[kernels2] (0,0)  -- (1,1) node[xiesf] {};
\draw[kernels2] (0,0) node[not] {} -- (0,1.5) node[xi] {}; }

\DeclareSymbol{Xi4eac2}{1.5}{\draw (-1,2.5) node[xic] {} -- (-1,1) node[xic] {} -- (0,0); %73
 \draw[kernels2] (0,0)  -- (1,1) node[xi] {};
\draw[kernels2] (0,0) node[not] {} -- (0,1.5) node[xi] {}; }

\DeclareSymbol{Xi4eact1}{1.5}{\draw (-1,2.5) node[xic] {} -- (-1,1) node[xi] {} -- (0,0); %74
 \draw (0,0)  -- (1,1) node[xic] {};
\draw[rho] (0,0) node[not] {} -- (0,1.5) node[xi] {}; }

\DeclareSymbol{Xi4eact2}{1.5}{\draw[rho] (-1,2.5) node[xic] {} -- (-1,1) node[xi] {} -- (0,0); %75
 \draw (0,0)  -- (1,1) node[xic] {};
\draw (0,0) node[not] {} -- (0,1.5) node[xi] {}; }

\DeclareSymbol{Xi4eabis}{1.5}{\draw (-1,2.5) node[xi] {} -- (-1,1) ; \draw[kernels2] (-1,1) node[xi] {} -- (0,0); 
 \draw (0,0)  -- (1,1) node[xi] {};%76
\draw[kernels2] (0,0) node[not] {} -- (0,1.5) node[xi] {}; }

\DeclareSymbol{Xi4eabisc1}{1.5}{\draw (-1,2.5) node[xic] {} -- (-1,1) ; \draw[kernels2] (-1,1) node[xi] {} -- (0,0); 
 \draw (0,0)  -- (1,1) node[xi] {};%77
\draw[kernels2] (0,0) node[not] {} -- (0,1.5) node[xic] {}; }

\DeclareSymbol{Xi4eabisc1b}{1.5}{\draw (-1,2.5) node[xic] {} -- (-1,1) ; \draw[kernels2] (-1,1) node[xi] {} -- (0,0); 
 \draw (0,0)  -- (1,1) node[xi] {};%77
\draw[kernels2] (0,0) node[not] {} -- (0,1.5) node[xiesf] {}; }

\DeclareSymbol{Xi4eabisc1bis}{1.5}{\draw (-1,2.5) node[xi] {} -- (-1,1) ; \draw[kernels2] (-1,1) node[xi] {} -- (0,0); 
 \draw (0,0)  -- (1,1) node[xi] {};%78
\draw[kernels2] (0,0) node[not] {} -- (0,1.5) node[xi] {};
\draw (-2,1) node[] {\tiny{$i$}};
\draw (-2,2.5) node[] {\tiny{$\ell$}};
\draw (2,1) node[] {\tiny{$k$}};
\draw (0,2.5) node[] {\tiny{$j$}};
 }

\DeclareSymbol{Xi4eabisc1tris}{1.5}{\draw (-1,2.5) node[xi] {} -- (-1,1) ; \draw[kernels2] (-1,1) node[xi] {} -- (0,0); 
 \draw (0,0)  -- (1,1) node[xi] {};%78
\draw[kernels2] (0,0) node[not] {} -- (0,1.5) node[xi] {};
\draw (-2,1) node[] {\tiny{i}};
\draw (-2,2.5) node[] {\tiny{j}};
\draw (2,1) node[] {\tiny{j}};
\draw (0,2.5) node[] {\tiny{i}};
 }

\DeclareSymbol{Xi4eabisc1quater}{1.5}{\draw (-1,2.5) node[xic] {} -- (-1,1) ; \draw[kernels2] (-1,1) node[xi] {} -- (0,0); 
 \draw (0,0)  -- (1,1) node[xic] {};%78
\draw[kernels2] (0,0) node[not] {} -- (0,1.5) node[xi] {};
 }

\DeclareSymbol{Xi4eabisc2}{1.5}{\draw (-1,2.5) node[xic] {} -- (-1,1) ; \draw[kernels2] (-1,1) node[xi] {} -- (0,0); %79
 \draw (0,0)  -- (1,1) node[xic] {};
\draw[kernels2] (0,0) node[not] {} -- (0,1.5) node[xi] {}; }

\DeclareSymbol{Xi4eabisc2l}{1.5}{\draw (-1,2.5) node[xiesf] {} -- (-1,1) ; \draw[kernels2] (-1,1) node[xi] {} -- (0,0); %79
 \draw (0,0)  -- (1,1) node[xic] {};
\draw[kernels2] (0,0) node[not] {} -- (0,1.5) node[xi] {}; }

\DeclareSymbol{Xi4eabisc2r}{1.5}{\draw (-1,2.5) node[xic] {} -- (-1,1) ; \draw[kernels2] (-1,1) node[xi] {} -- (0,0); %79
 \draw (0,0)  -- (1,1) node[xiesf] {};
\draw[kernels2] (0,0) node[not] {} -- (0,1.5) node[xi] {}; }

\DeclareSymbol{Xi4eabisc3}{1.5}{\draw (-1,2.5) node[xic] {} -- (-1,1) ; \draw[kernels2] (-1,1) node[xic] {} -- (0,0); %80
 \draw (0,0)  -- (1,1) node[xi] {};
\draw[kernels2] (0,0) node[not] {} -- (0,1.5) node[xi] {}; }

\DeclareSymbol{Xi4eb}{0}{%81
\draw[kernels2] (0,2) node[xi] {} -- (-1,1) ; \draw[kernels2] (-2,2)  node[xi] {} -- (-1,1) ; \draw (-1,1)  node[not] {} -- (0,0); 
 \draw (0,0) node[xi] {}  -- (1,1) node[xi] {};
}

\DeclareSymbol{Xi4eab}{1.5}{\draw[kernels2] (-1,2.5) node[xi] {} -- (-1,1) ; \draw[kernels2] (-2,2)  node[xi] {} -- (-1,1) ; \draw (-1,1)  node[not] {} -- (0,0); %82
 \draw[kernels2] (0,0)  -- (1,1) node[xi] {};
\draw[kernels2] (0,0) node[not] {} -- (0,1.5) node[xi] {}; 
}

\DeclareSymbol{Xi4eabdis}{1.5}{\draw[kernels2] (-1,2.5) node[xies] {} -- (-1,1) ; \draw[kernels2] (-2,2)  node[xies] {} -- (-1,1) ; \draw (-1,1)  node[not] {} -- (0,0); %82
 \draw[kernels2] (0,0)  -- (1,1) node[xies] {};
\draw[kernels2] (0,0) node[not] {} -- (0,1.5) node[xies] {}; 
}

\DeclareSymbol{Xi4eabc1}{1.5}{\draw[kernels2] (-1,2.5) node[xic] {} -- (-1,1) ; \draw[kernels2] (-2,2)  node[xi] {} -- (-1,1) ; \draw (-1,1)  node[not] {} -- (0,0); %83
 \draw[kernels2] (0,0)  -- (1,1) node[xic] {};
\draw[kernels2] (0,0) node[not] {} -- (0,1.5) node[xi] {}; 
}

\DeclareSymbol{Xi4eabc2}{1.5}{\draw[kernels2] (-1,2.5) node[xi] {} -- (-1,1) ; \draw[kernels2] (-2,2)  node[xi] {} -- (-1,1) ; \draw (-1,1)  node[not] {} -- (0,0); %84
 \draw[kernels2] (0,0)  -- (1,1) node[xic] {};
\draw[kernels2] (0,0) node[not] {} -- (0,1.5) node[xic] {}; 
}

\DeclareSymbol{Xi4eabbis}{1.5}{\draw[kernels2] (-1,2.5) node[xi] {} -- (-1,1) ; \draw[kernels2] (-2,2)  node[xi] {} -- (-1,1) ; \draw[kernels2] (-1,1)  node[not] {} -- (0,0); %85
 \draw (0,0)  -- (1,1) node[xi] {};
\draw[kernels2] (0,0) node[not] {} -- (0,1.5) node[xi] {}; 
}

\DeclareSymbol{Xi4eabbisc1}{1.5}{\draw[kernels2] (-1,2.5) node[xic] {} -- (-1,1) ; \draw[kernels2] (-2,2)  node[xi] {} -- (-1,1) ; \draw[kernels2] (-1,1)  node[not] {} -- (0,0); %86
 \draw (0,0)  -- (1,1) node[xic] {};
\draw[kernels2] (0,0) node[not] {} -- (0,1.5) node[xi] {}; 
}

\DeclareSymbol{Xi4eabbisc1perm}{1.5}{\draw[kernels2] (-1,2.5) node[xi] {} -- (-1,1) ; \draw[kernels2] (-2,2)  node[xic] {} -- (-1,1) ; \draw[kernels2] (-1,1)  node[not] {} -- (0,0); %86
 \draw (0,0)  -- (1,1) node[xic] {};
\draw[kernels2] (0,0) node[not] {} -- (0,1.5) node[xi] {}; 
}

\DeclareSymbol{Xi4eabbisc2}{1.5}{\draw[kernels2] (-1,2.5) node[xi] {} -- (-1,1) ; \draw[kernels2] (-2,2)  node[xi] {} -- (-1,1) ; \draw[kernels2] (-1,1)  node[not] {} -- (0,0); %87
 \draw (0,0)  -- (1,1) node[xic] {};
\draw[kernels2] (0,0) node[not] {} -- (0,1.5) node[xic] {}; 
}

\DeclareSymbol{Xi2cbis}{0}{\draw[kernels2] (0,1) -- (0.8,2.2) node[xi] {};\draw[kernels2] (0,-0.25) node[not] {} -- (0,1); \draw[kernels2] (0,1) node[not] {} -- (-0.8,2.2) node[xi] {};}%88

\DeclareSymbol{Xi2cbis1}{0}{\draw (0,1) -- (-0.8,2.2) node[xi] {};\draw[kernels2] (0,-0.25) node[not] {} -- (0,1) node[xi] {}; }%89

\DeclareSymbol{Xi2Xbis}{-2}{\draw[kernels2] (0,-0.25)  -- (-1,1) ; \draw (-1,1) node[xix] {};%91
\draw[kernels2] (0,-0.25) node[not] {} -- (1,1) node[xi] {};}

\DeclareSymbol{XXi2bis}{-2}{\draw[kernels2] (0,-0.25) -- (-1,1) node[xi] {};%92
\draw[kernels2] (0,-0.25) node[X] {} -- (1,1) node[xi] {};}

\DeclareSymbol{I1XiIXi}{0}{\draw[kernels2] (0,-0.25) -- (1,1) node[xi] {};%93
\draw (0,-0.25) node[not] {} -- (-1,1) node[xi] {};}

\DeclareSymbol{I1XiIXib}{0}{\draw  (0,-0.25) node[xi] {} -- (0,1) node[not] {};
\draw[kernels2] (0,1) -- (0,2.25) ; \draw (0,2.25) node[xi]{}; }%94

\DeclareSymbol{I1XiIXic}{0}{%95
\draw[kernels2] (0,0) -- (1,1) node[xi] {} ; 
\draw[kernels2] (0,0) node[not] {}  -- (-1,1) node[not] {} -- (0,2) node[xi] {};
}

\DeclareSymbol{thin}{1.4}{\draw[pagebackground] (-0.3,0) -- (0.3,0); \draw  (0,0) -- (0,2);}
\DeclareSymbol{thin2}{1.4}{\draw[pagebackground] (-0.3,0) -- (0.3,0); \draw[tinydots]  (0,0) -- (0,2);}

\DeclareSymbol{thick}{1.4}{\draw[pagebackground] (-0.3,0) -- (0.3,0); \draw[kernels2]  (0,0) -- (0,2);}

\DeclareSymbol{thick2}{1.4}{\draw[pagebackground] (-0.3,0) -- (0.3,0); \draw[kernels2,tinydots]  (0,0) -- (0,2);}

\DeclareSymbol{Xi4ind}{2}{\draw (0,0) node[xi,label={[label distance=-0.2em]right: \scriptsize  $ i $}]  { } -- (-1,1) node[xi,label={[label distance=-0.2em]left: \scriptsize  $ j $}] {} -- (0,2) node[xi,label={[label distance=-0.2em]right: \scriptsize  $ k $}] {} -- (-1,3) node[xi,label={[label distance=-0.2em]left: \scriptsize  $ \ell $}] {};}%115

\DeclareSymbol{Xi4c1}{2}{\draw (0,0) node[xic] {} -- (-1,1) node[xi] {} -- (0,2) node[xic] {} -- (-1,3) node[xi] {};} %119
\DeclareSymbol{IXi2ex}{0}{\draw (0,-0.25) node[xie] {} -- (-1,1) node[xi] {} ; \draw (0,-0.25)-- (1,1) node[xi] {};}
\DeclareSymbol{IXi2ex1}{0}{\draw (0,-0.25) node[xie] {} -- (-1,1) node[xi] {} -- (0,2) node[xi] {};}

\DeclareSymbol{Xi4b1}{0}{\draw(0,1.5) node[xic] {} -- (0,0); \draw (-1,1) node[xic] {} -- (0,0) node[xi] {} -- (1,1) node[xi] {};}

\DeclareSymbol{Xi4ec1}{0}{\draw (0,2) node[xi] {} -- (-1,1) node[xic] {} -- (0,0) node[xic] {} -- (1,1) node[xi] {};}
\DeclareSymbol{Xi4ec2}{0}{\draw (0,2) node[xic] {} -- (-1,1) node[xi] {} -- (0,0) node[xic] {} -- (1,1) node[xi] {};}
\DeclareSymbol{Xi4ec3}{0}{\draw (0,2) node[xic] {} -- (-1,1) node[xic] {} -- (0,0) node[xi] {} -- (1,1) node[xi] {};}

\DeclareSymbol{I1Xi4ac1}{2}{\draw[kernels2] (0,0) node[not] {} -- (-1,1) ; \draw[kernels2] (0,0) node[not] {} -- (1,1) node[xic] {} ;
\draw (-1,1) node[xi] {} -- (0,2) node[xic] {} -- (-1,3) node[xi] {};}%161

\DeclareSymbol{I1Xi4ac2}{2}{\draw[kernels2] (0,0) node[not] {} -- (-1,1) ; \draw[kernels2] (0,0) node[not] {} -- (1,1) node[xic] {} ;
\draw (-1,1) node[xi] {} -- (0,2) node[xi] {} -- (-1,3) node[xic] {};}

\DeclareSymbol{I1Xi4bp}{2}{\draw (0,0) node[not] {} -- (-1,1) node[xi] {} -- (0,2) ; \draw[kernels2] (0,2) node[not] {} -- (-1,3) node[xi] {};\draw[kernels2] (0,2)  -- (1,3) node[xi] {};
}

\DeclareSymbol{I1Xi4bc1}{2}{\draw (0,0) node[xic] {} -- (-1,1) node[xi] {} -- (0,2) ; \draw[kernels2] (0,2) node[not] {} -- (-1,3) node[xi] {};\draw[kernels2] (0,2)  -- (1,3) node[xic] {};
}%165

\DeclareSymbol{I1Xi4bc2}{2}{\draw (0,0) node[xic] {} -- (-1,1) node[xi] {} -- (0,2) ; \draw[kernels2] (0,2) node[not] {} -- (-1,3) node[xic] {};\draw[kernels2] (0,2)  -- (1,3) node[xi] {};
}

\DeclareSymbol{I1Xi4cp}{2}{\draw (0,0) node[not] {} -- (-1,1) node[not] {}; \draw[kernels2] (-1,1) -- (0,2) ; 
\draw[kernels2] (-1,1) -- (-2,2) node[xi] {} ;
\draw (0,2) node[xi] {} -- (-1,3) node[xi] {};}

\DeclareSymbol{I1Xi4cc1}{2}{\draw (0,0) node[xic] {} -- (-1,1) node[not] {}; \draw[kernels2] (-1,1) -- (0,2) ; 
\draw[kernels2] (-1,1) -- (-2,2) node[xi] {} ;
\draw (0,2) node[xic] {} -- (-1,3) node[xi] {};}%169

\DeclareSymbol{I1Xi4cc2}{2}{\draw (0,0) node[xic] {} -- (-1,1) node[not] {}; \draw[kernels2] (-1,1) -- (0,2) ; 
\draw[kernels2] (-1,1) -- (-2,2) node[xi] {} ;
\draw (0,2) node[xi] {} -- (-1,3) node[xic] {};}

\DeclareSymbol{I1Xi4abc1}{2}{\draw[kernels2] (0,0) node[not] {} -- (-1,1) ; \draw[kernels2] (0,0) node[not] {} -- (1,1) node[xic] {};\draw (-1,1) node[xi] {} -- (0,2) ; \draw[kernels2] (0,2) node[not] {} -- (-1,3) node[xic] {};\draw[kernels2] (0,2)  -- (1,3) node[xi] {}; }

\DeclareSymbol{I1Xi4abc2}{2}{\draw[kernels2] (0,0) node[not] {} -- (-1,1) ; \draw[kernels2] (0,0) node[not] {} -- (1,1) node[xic] {};\draw (-1,1) node[xi] {} -- (0,2) ; \draw[kernels2] (0,2) node[not] {} -- (-1,3) node[xi] {};\draw[kernels2] (0,2)  -- (1,3) node[xic] {}; }%173

\DeclareSymbol{R1}{0}{\draw (-1,1) node[xi] {} -- (0,0) node[not] {};%131
\draw[kernels2] (0,1.5) node[xic] {} -- (0,0) -- (1,1) node[xic] {};}
\DeclareSymbol{R2}{0}{\draw (-1,1) node[xic] {} -- (0,0) node[not] {};
\draw[kernels2] (0,1.5)  {} -- (0,0) -- (1,1)  {};
\draw (0,1.5) node[xi] {};
\draw (1,1) node[xic] {};
}%132
\DeclareSymbol{R3}{1}{\draw[kernels2] (-1,1.5)  {} -- (0,0) node[not] {} -- (1,1.5);%133
\draw (-1,1.5) node[xi] {};
\draw[kernels2] (0,3) {} -- (1,1.5) -- (2,3)  {};
\draw  (0,3) node[xic] {} ;
\draw (2,3) node[xic] {};}
\DeclareSymbol{R4}{1}{\draw[kernels2] (-1,1.5) node[xic] {} -- (0,0) node[not] {} -- (1,1.5);%134
\draw[kernels2] (0,3) {} -- (1,1.5) -- (2,3) node[xic] {};
\draw (0,3) node[xi] {};}

\DeclareSymbol{I1Xi4bcp}{2}{\draw (0,0) node[not] {} -- (-1,1) node[not] {}; \draw[kernels2] (-1,1) -- (0,2) ; %175
\draw[kernels2] (-1,1) -- (-2,2) node[xi] {} ; \draw[kernels2] (0,2) node[not] {} -- (-1,3) node[xi] {};\draw[kernels2] (0,2)  -- (1,3) node[xi] {};
}

\DeclareSymbol{I1Xi4bcc1}{2}{\draw (0,0) node[xic] {} -- (-1,1) node[not] {}; \draw[kernels2] (-1,1) -- (0,2) ; 
\draw[kernels2] (-1,1) -- (-2,2) node[xi] {} ; \draw[kernels2] (0,2) node[not] {} -- (-1,3) node[xi] {};\draw[kernels2] (0,2)  -- (1,3) node[xic] {};
}

\DeclareSymbol{I1Xi4bcc2}{2}{\draw (0,0) node[xic] {} -- (-1,1) node[not] {}; \draw[kernels2] (-1,1) -- (0,2) ; 
\draw[kernels2] (-1,1) -- (-2,2) node[xi] {} ; \draw[kernels2] (0,2) node[not] {} -- (-1,3) node[xic] {};\draw[kernels2] (0,2)  -- (1,3) node[xi] {};
} %177

\DeclareSymbol{2I1Xi4bc1}{2}{\draw[kernels2] (0,0) node[not] {} -- (-1,1) ;
\draw[kernels2] (0,0) -- (1,1);
\draw (-1,1) node[xic] {} -- (-1,2.5) node[xi] {};
\draw (1,1)  node[xic] {} -- (1,2.5) node[xi] {};
}

\DeclareSymbol{2I1Xi4bc2}{2}{\draw[kernels2] (0,0) node[not] {} -- (-1,1) ;
\draw[kernels2] (0,0) -- (1,1);
\draw (-1,1) node[xi] {} -- (-1,2.5) node[xic] {};
\draw (1,1)  node[xic] {} -- (1,2.5) node[xi] {};
}

\DeclareSymbol{diff2I1Xi4bc2}{2}{\draw (0,0) node[diff] {} -- (-1,1) ;
\draw (0,0) -- (1,1);
\draw (-1,1) node[xi] {} -- (-1,2.5) node[xic] {};
\draw (1,1)  node[xic] {} -- (1,2.5) node[xi] {};
}

\DeclareSymbol{2I1Xi4bc3}{2}{\draw[kernels2] (0,0) node[not] {} -- (-1,1) ;
\draw[kernels2] (0,0) -- (1,1);
\draw (-1,1) node[xic] {} -- (-1,2.5) node[xic] {};
\draw (1,1)  node[xi] {} -- (1,2.5) node[xi] {};
}

\DeclareSymbol{Xi41}{0}{\draw (0,1) -- (0.8,2.2) node[xic] {};\draw (0,-0.25) node[xi] {} -- (0,1) node[xi] {} -- (-0.8,2.2) node[xic] {};} 

\DeclareSymbol{Xi42}{0}{\draw (0,1) -- (0.8,2.2) node[xi] {};\draw (0,-0.25) node[xic] {} -- (0,1) node[xi] {} -- (-0.8,2.2) node[xic] {};}

\DeclareSymbol{Xi4ca1}{0}{\draw (0,1) -- (-1,2.2) node[xic] {};\draw (0,-0.25) node[xi] {} -- (0,1) ; \draw[kernels2] (0,1) node[not] {} -- (1,2.2) node[xic] {};
\draw[kernels2] (0,1) {} -- (0,2.7) node[xi] {};
}

\DeclareSymbol{Xi4ca2}{0}{\draw (0,1) -- (-1,2.2) node[xi] {};\draw (0,-0.25) node[xi] {} -- (0,1) ; \draw[kernels2] (0,1) node[not] {} -- (1,2.2) node[xic] {};
\draw[kernels2] (0,1) {} -- (0,2.7) node[xic] {};
}

\DeclareSymbol{Xi4cap}{0}{\draw (0,1) -- (-1,2.2) node[xi] {};\draw (0,-0.25) node[not] {} -- (0,1) ; \draw[kernels2] (0,1) node[not] {} -- (1,2.2) node[xi] {};
\draw[kernels2] (0,1) {} -- (0,2.7) node[xi] {};
}

\DeclareSymbol{Xi3a}{0}{
 \draw (-1,1)  node[xi] {} -- (0,0); 
 \draw (0,0) node[xi] {}  -- (1,1) node[xi] {};
}

\DeclareSymbol{Xi4ebc1}{0}{
\draw[kernels2] (0,2) node[xi] {} -- (-1,1) ; \draw[kernels2] (-2,2)  node[xic] {} -- (-1,1) ; \draw (-1,1)  node[not] {} -- (0,0); 
 \draw (0,0) node[xic] {}  -- (1,1) node[xi] {};
}

\DeclareSymbol{Xi4ebc2}{0}{
\draw[kernels2] (0,2) node[xi] {} -- (-1,1) ; \draw[kernels2] (-2,2)  node[xi] {} -- (-1,1) ; \draw (-1,1)  node[not] {} -- (0,0); 
 \draw (0,0) node[xic] {}  -- (1,1) node[xic] {};
}

\DeclareSymbol{Xi2cbispex}{0}{\draw[kernels2] (0,1) -- (0.8,2.2) node[xi] {};\draw (0,-0.25) node[xie] {} -- (0,1); \draw[kernels2] (0,1) node[not] {} -- (-0.8,2.2) node[xi] {};}

\DeclareSymbol{Xi2cbis1p}{0}{\draw (0,1) -- (-0.8,2.2) node[xi] {};\draw (0,-0.25) node[not] {} -- (0,1) node[xi] {}; }

\DeclareSymbol{Xi2Xp}{-2}{\draw (0,-0.25) node[not] {} -- (-1,1) node[xix] {};} % 229 not used

\DeclareSymbol{I1XiIXib}{0}{\draw  (0,-0.25) node[xi] {} -- (0,1) node[not] {};
\draw[kernels2] (0,1) -- (0,2.25) ; \draw (0,2.25) node[xi]{}; }

\DeclareSymbol{IXi2b}{0}{\draw  (0,-0.25) node[xi] {} -- (0,1) node[not] {};
\draw (0,1) -- (0,2.25) ; \draw (0,2.25) node[xi]{}; }

\DeclareSymbol{IXi2bex}{0}{\draw  (0,-0.25) node[xi] {} -- (0,1) node[xie] {};
\draw (0,1) -- (0,2.25) ; \draw (0,2.25) node[xi]{}; }

 \def\1{\mathbf{\symbol{1}}}

\DeclareSymbol{diff}{0}{
\draw (0,0.5) node[diff] {};
}

\DeclareSymbol{diff1}{0}{
\draw (0,0.5) node[diff1] {};
}

\DeclareSymbol{diff2}{0}{
\draw (0,0.5) node[diff2] {};
}

\DeclareSymbol{geo}{0}{
\draw (0,0) node[diff] {};
\draw (0.3,0) node[diff] {};
}

\DeclareSymbol{generic}{0}{
\draw (0,0.6) node[xi] {};
}

\DeclareSymbol{g}{0}{
\draw (0,0.6) node[g] {};
}

\DeclareSymbol{Ito}{0}{
\draw (0,0.6) node[xies] {};
}

\DeclareSymbol{Itob}{0}{
\draw (0,0.6) node[xiesf] {};
}

\DeclareSymbol{greycirc}{0}{
\draw (0,0.3) node[xi] {};
}

\DeclareSymbol{not}{0}{
\draw (0,0.6) node[not] {};
\draw[tinydots] (0,0.6) circle (0.8);
}

\DeclareSymbol{genericb}{0}{
\draw (0,0.6) node[xic] {};
}

\DeclareSymbol{bluecirc}{0}{
\draw (0,0.3) node[xic] {};
}

\DeclareSymbol{genericxix}{0}{
\draw (0,0.6) node[xix] {};
}

\DeclareSymbol{genericX}{0}{
\draw (0,0.6) node[X] {};
}

\DeclareSymbol{diffIto}{1}{
\draw  (0,2.5) -- (0,0) ;
\draw (0,-0.1) node[diff] {};
\draw (0,2.5) node[xies] {};
}
\DeclareSymbol{Itodiff}{2}{
\draw(0,2.9) -- (0,-0.2);
\draw (0,2.9) node[diff] {};
\draw (0,-0.1) node[xies] {};
}

\DeclareSymbol{diffgeneric}{1}{
\draw  (0,2.5) -- (0,0) ;
\draw (0,-0.1) node[diff] {};
\draw (0,2.5) node[xi] {};
}

\DeclareSymbol{genericdiff}{2}{
\draw(0,2.9) -- (0,-0.2);
\draw (0,2.9) node[diff] {};
\draw (0,-0.1) node[xi] {};
}

\DeclareSymbol{diffdot}{2}{
\draw  (0,3) -- (0,-0.1) ;
\draw (0,3) node[not] {};
\draw (0,-0.1) node[diff] {};
}

\DeclareSymbol{diffdotmini}{0}{
\draw  (0,0) -- (0,1.2) ;
\draw (0,1.2) node[not] {};
\draw (0,0) node[diffmini] {};
}

\DeclareSymbol{dotdiff}{2}{
\draw[kernelsmod]  (0,3) -- (0,-0.1) ;
\draw (0,3) node[diff] {};
\draw (0,-0.1) node[not] {};
}

\DeclareSymbol{dotdiff1}{2}{
\draw[kernelsmod]  (0,3) -- (0,-0.1) ;
\draw (0,3) node[diff1] {};
\draw (0,-0.1) node[not] {};
}

\DeclareSymbol{dotdiff1mini}{0}{
\draw[kernelsmod]  (0,1.2) -- (0,0) ;
\draw (0,1.2) node[diffmini] {};
\draw (0,0) node[not] {};
}

\DeclareSymbol{dotdiff2}{2}{
\draw (0,3) -- (0,-0.1) ;
\draw (0,3) node[diff] {};
\draw (0,-0.1) node[not] {};
}

\DeclareSymbol{dotdiff2mini}{0}{
\draw (0,1.2) -- (0,0) ;
\draw (0,1.2) node[diffmini] {};
\draw (0,0) node[not] {};
}

\DeclareSymbol{dotdiffstraight}{0}{
\draw  (0,3) -- (0,-0.1) ;
\draw (0,3) node[diff] {};
\draw (0,-0.1) node[not] {};
}

\DeclareSymbol{arbre1}{0}{
\draw  (0,0) -- (1.5,1.5) ;
\draw (1.5,1.5) node[not] {};
\draw (0,0) node[not] {};
}

\DeclareSymbol{arbre2}{0}{
\draw  (0,0) -- (1.5,1.5) ;
\draw[kernelsmod] (0,0) -- (-1.5,1.5);
\draw (1.5,1.5) node[not] {};
\draw (0,0) node[not] {};
\draw (-1.5,1.5) node[xi] {};
}

\DeclareSymbol{arbre3}{0}{
\draw  (0,0) -- (1.5,1.5) ;
\draw[kernelsmod] (1.5,1.5) -- (0,3);
\draw (0,0) node[not] {};
\draw (1.5,1.5) node[not] {};
\draw (0,3) node[xi] {};
}

\DeclareSymbol{treeeval}{0}{
\draw (0,0) -- (1,1);
\draw (0,0) node[xi] {};
\draw (1.25,1.25) node[xi] {};
\draw (-0.6,0.6) node[]{\tiny{$i$}};
\draw (0.65,1.85) node[]{\tiny{$j$}};
}

\DeclareSymbol{testeval}{0}{
\draw (0,0) -- (1,1);
\draw (0,0) -- (-1,1);
\draw (0,0) node[xi] {};
\draw (1.25,1.25) node[xi] {};
\draw (-1.25,1.25) node[xi] {};
\draw (-0.6,-0.6) node[]{\tiny{$i$}};
\draw (0.65,1.85) node[]{\tiny{$j$}};
\draw (-1.95,1.85) node[]{\tiny{$k$}};
}

\DeclareSymbol{treeeval2}{0}{
\draw[kernelsmod] (-0.25,-1) -- (1,0.5) ;
\draw[kernelsmod] (1,0.5) -- (-0.25,2);
\draw (1,0.5) node[diff2] {};
\draw (-0.25,-1) node[not] {};
\draw (-0.25,2) node[xi] {};
\draw (-0.6,1.2) node[]{\tiny{1}};
}

\DeclareSymbol{arbreact}{1}{
\draw (0,0) node[not] {};
\draw[kernelsmod] (0,0) -- (1,1);
\draw[kernelsmod] (0,0) -- (-1,1);
\draw (-1,1) node[xic] {};
\draw  (0,2) -- (1,1) ;
\draw (0,2) node[xic] {};
\draw (1,1) node[xi] {};
}

\DeclareSymbol{arbreact1}{0}{
\draw (0,-1.5) -- (0,0);
\draw[kernelsmod] (0,0) -- (1,1);
\draw[kernelsmod] (0,0) -- (-1,1);
\draw  (0,2) -- (1,1) ;
\draw (0,-1.5) node[diff] {};
\draw (0,0) node[not] {};
\draw (-1,1) node[xic] {};
\draw (0,2) node[xic] {};
\draw (1,1) node[xi] {};
}

\DeclareSymbol{arbreact2}{0}{
\draw (0,-0.75) -- (-1,0.5); 
\draw (0,-0.75) -- (1,0.5);
\draw (0,1.5) -- (1,0.5);
\draw (0,1.5) node[xic] {};
\draw (1,0.5) node[xi] {};
\draw (-1,0.5) node[xic] {};
\draw (0,-0.75) node[diff] {};
}

\DeclareSymbol{arbreact3}{0}{
\draw[kernelsmod] (0,-0.75) -- (-1,0.5); 
\draw[kernelsmod] (0,-0.75) -- (1,0.5);
\draw (0,1.75) -- (1,0.5);
\draw (2,1.75) -- (1,0.5);
\draw (0,1.75) node[xic] {};
\draw (1,0.5) node[diff] {};
\draw (-1,0.5) node[xic] {};
\draw (2,1.75) node[xi] {};
\draw (0,-0.75) node[not] {};
}

\DeclareSymbol{pre_im_I1Xitwo}{0}{
\draw[kernels2] (0,-0.3) node[not] {} -- (-0.6,0.7) ;
\draw[kernels2] (0,-0.3) -- (0.6,0.7);
\draw (0,0.9) node[g] {};
}

\DeclareSymbol{pre_im_cI1Xi4ab}{2}{
\draw[kernels2] (0,-1) node[not] {} -- (-0.6,0) ;
\draw[kernels2] (0,-1) -- (0.6,0);
\draw (0,0.2) node[g] {};
\draw (0,0.6) -- (0,1.5);
\draw[kernels2] (0,1.5) node[not] {} -- (-0.6,2.5) ;
\draw[kernels2] (0,1.5) -- (0.6,2.5);
\draw (0,2.7) node[g] {};
}%46

\DeclareSymbol{pre_im_I1Xi4acc2}{0}{
\draw[kernels2] (-1,-0.5) node[not] {} -- (-1.6,0.5) ;
\draw[kernels2] (-1,-0.5) -- (-0.4,0.5);
\draw[kernels2] (-1,-0.5) -- (0.2,-1.5) node[not] {} ;
\draw (-1,1.1) -- (-1,2);
\draw[kernels2] (0.2,-1.5) -- (0.2,2);
\draw (-1,0.7) node[g] {};
\draw (-0.3,2.2) node[g] {};
}

\DeclareSymbol{pre_im_I1Xi4abcc2}{2}{
\draw[kernels2] (0,-1) node[not] {} -- (-1,0) node[not] {};
\draw[kernels2] (-1,1.2) node[not] {} -- (-1,0);
\draw[kernels2] (-1,1.2) -- (-1.5,2.5);
\draw[kernels2] (-1,1.2) -- (-0.5,2.5);
\draw[kernels2] (-1,0) -- (0.7,2.5);
\draw[kernels2] (0,-1) -- (1.5,2.5);
\draw (-1,2.7) node[g] {};
\draw (1,2.7) node[g] {};
}

\DeclareSymbol{pre_im_2I1Xi4c1}{2}{
\draw[kernels2] (0,-0.5) node[not] {} -- (-1,0.5) node[not] {};
\draw[kernels2] (0,-0.5) -- (1,0.5) node[not] {};
\draw[kernels2] (-1,0.5) node[not] {}-- (-1.7,2);
\draw[kernels2]  (-1,2) -- (1,0.5);
\draw[kernels2] (-1,0.5) -- (1,2);
\draw[kernels2] (1,0.5) -- (1.7,2);
\draw (-1.2,2.2) node[g] {};
\draw (1.2,2.2) node[g] {};
}

\DeclareSymbol{pre_im_Xi4eabisc2}{2}{
\draw[kernels2] (1.2,-0.5) node[not] {} -- (-0.7,0.8) ;
\draw[kernels2] (1.2,-0.5) -- (0.4,0.8);
\draw (0,1.4)  -- (0,2.2);
\draw (1.2,2.2) -- (1.2,-0.6);
\draw (0,1) node[g] {};
\draw (0.6,2.4) node[g] {};
}

\DeclareSymbol{pre_im_Xi4eabisc22}{2}{
\draw (1.2,-0.5) node[not] {} -- (-0.7,0.8) ;
\draw[kernels2] (1.2,-0.5) -- (0.4,0.8);
\draw (0,1.4)  -- (0,2.2);
\draw[kernels2] (1.2,2.2) -- (1.2,-0.6);
\draw (0,1) node[g] {};
\draw (0.6,2.4) node[g] {};
}

\DeclareSymbol{pre_im_Xi4eabisc222}{2}{
\draw[kernels2] (0.4,-0.5) node[not] {} -- (-0.6,1) ;
\draw[kernels2] (1.2,0) -- (0.3,1);
\draw (0,1.1)  -- (0,2.5);
\draw[kernels2] (1.2,2.5) -- (1.2,0) node[not] {} -- (0.4,-0.6);
\draw (0,1.2) node[g] {};
\draw (0.6,2.5) node[g] {};
}

\DeclareSymbol{pre_im_Xi4eabc2}{2}{
\draw (0,-0.5) node[not] {} -- (-1,0.5) node[not] {};
\draw[kernels2] (-1,0.5) -- (-1.5,2);
\draw[kernels2] (0,-0.5)  -- (0.7,2);
\draw[kernels2] (-1,0.5) -- (-0.5,2);
\draw[kernels2] (0,-0.5) -- (1.5,2);
\draw (-1,2.2) node[g] {};
\draw (1,2.2) node[g] {};
}

\DeclareSymbol{pre_im_Xi4eabbisc2}{2}{
\draw[kernels2] (0,-0.5) node[not] {} -- (-1,0.5) node[not] {};
\draw[kernels2] (-1,0.5) -- (-1.5,2);
\draw[kernels2] (0,-0.5)  -- (0.7,2);
\draw[kernels2] (-1,0.5) -- (-0.5,2);
\draw (0,-0.5) -- (1.5,2);
\draw (-1,2.2) node[g] {};
\draw (1,2.2) node[g] {};
}

\DeclareSymbol{pre_im_I1Xi4abcc1}{2}{
\draw[kernels2] (0,-1) node[not] {} -- (-1,0) node[not] {};
\draw[kernels2] (0,1.1) node[not] {} -- (-1,0);
\draw[kernels2] (-1,0) -- (-1.5,2.5);
\draw[kernels2] (0,1.1) node[not] {} -- (-0.5,2.5);
\draw[kernels2] (0,1.1) -- (0.5,2.5);
\draw[kernels2] (0,-1) -- (1.5,2.5);
\draw (-1,2.7) node[g] {};
\draw (1,2.7) node[g] {};
}

\DeclareSymbol{pre_im_Xi4eabc1}{2}{
\draw (0,-0.5) node[not] {} -- (-1,0.5) node[not] {};
\draw[kernels2] (-1,0.5) -- (-1.5,2);
\draw[kernels2] (0,-0.5)  -- (-0.5,2);
\draw[kernels2] (-1,0.5) -- (0.8,2);
\draw[kernels2] (0,-0.5) -- (1.5,2);
\draw (-1,2.2) node[g] {};
\draw (1,2.2) node[g] {};
}

\DeclareSymbol{pre_im_Xi4ba1b}{2}{
\draw[kernels2] (0,0) node[not] {}  -- (1.8,1.5);
\draw[kernels2] (0,0) -- (0.8,1.5);
\draw (0,-0.1) -- (-1.8,1.5);
\draw (0,-0.1) -- (-0.8,1.5);
\draw (-1,1.7) node[g] {};
\draw (1,1.7) node[g] {};
}

\DeclareSymbol{pre_im_Xi4ba2}{2}{
\draw (0,-0.1) node[not] {}  -- (1.8,1.5);
\draw[kernels2] (0,0) -- (0.8,1.5);
\draw (0,-0.1) -- (-1.8,1.5);
\draw[kernels2] (0,0) -- (-0.8,1.5);
\draw (-1,1.7) node[g] {};
\draw (1,1.7) node[g] {};
}

\DeclareSymbol{pre_im_Xi4cabc2}{2}{
\draw[kernels2] (0,-0.5) node[not] {} -- (-1,0.5) node[not] {};
\draw[kernels2] (-1,0.5) -- (-1.5,2);
\draw (-1,0.5)  -- (0.7,2);
\draw[kernels2] (-1,0.5) -- (-0.5,2);
\draw[kernels2] (0,-0.5) -- (1.5,2);
\draw (-1,2.2) node[g] {};
\draw (1,2.2) node[g] {};
}

\DeclareSymbol{pre_im_Xi4cabc1}{2}{
\draw[kernels2] (0,-0.5) node[not] {} -- (-1,0.5) node[not] {};
\draw (-1,0.5) -- (-1.5,2);
\draw[kernels2] (-1,0.5)  -- (0.7,2);
\draw[kernels2] (-1,0.5) -- (-0.5,2);
\draw[kernels2] (0,-0.5) -- (1.5,2);
\draw (-1,2.2) node[g] {};
\draw (1,2.2) node[g] {};
}

\DeclareSymbol{pre_im_Xi4eabbisc1}{2}{
\draw[kernels2] (0,-0.5) node[not] {} -- (-1,0.5) node[not] {};
\draw[kernels2] (-1,0.5) -- (-1.5,2);
\draw[kernels2] (0,-0.5)  -- (-0.5,2);
\draw[kernels2] (-1,0.5) -- (0.8,2);
\draw (0,-0.5) -- (1.5,2);
\draw (-1,2.2) node[g] {};
\draw (1,2.2) node[g] {};
}

\DeclareSymbol{pre_im_1}{0}{
\draw[kernels2] (0,-0.5) node[not] {} -- (-0.6,0.5) ;
\draw[kernels2] (0,-0.5) -- (0.6,0.5);
\draw (0,1.1)  -- (-0.55,2);
\draw (0,1.1)  -- (0.55,2);
\draw (0,0.7) node[g] {};
\draw (0,2.2) node[g] {};
}

\DeclareSymbol{disconnect}{0}{
\draw[kernels2] (0,-0.5) node[not] {} -- (-0.6,0.5) ;
\draw[kernels2] (0,-0.5) -- (0.6,0.5);
\draw (-0.55,1.1)  -- (-0.55,2.3);
\draw (0.55,2.3) -- (0.55,1.5) -- (1.2,1.5) -- (1.2,3.5) -- (0.55,3.5) -- (0.55,2.7);
\draw (0,0.7) node[g] {};
\draw (0,2.5) node[g] {};
}

\DeclareSymbol{pre_im_2}{2}{\draw[kernels2] (0,0) node[not] {} -- (-1,1) node[not] {};
\draw[kernels2] (0,0) -- (1,1) node[not] {};
\draw[kernels2] (-1,1) -- (-1.5,2.5);
\draw[kernels2] (-1,1) -- (-0.5,2.5);
\draw[kernels2] (1,1) -- (0.5,2.5);
\draw[kernels2] (1,1) -- (1.5,2.5);
\draw (-1,2.7) node[g] {};
\draw (1,2.7) node[g] {};
}

\DeclareSymbol{CX_rec}{0}{
\draw [black] (-0.3,1) to (-0.3,-0.3);
\draw [black] (0.3,1) to (0.3,-0.3);
\draw [black] (-0.3,1) to (-0.3,2.3);
\draw [black] (0.3,1) to (0.3,2.3);
\draw (0,1) node[rec] {};
}

\DeclareSymbol{CX_cerc}{0}{
\draw [black] (0,1) to (0,-0.3);
\draw (0,1) node[cerc] {};
}

\DeclareMathAlphabet{\mathpzc}{OT1}{pzc}{m}{it}

\def\eqref#1{(\ref{#1})}

\makeatletter % Stolen from the internet to make a fat \cdot which isn't as fat as a \bullet
\newcommand*{\bigcdot}{}% Check if undefined
\DeclareRobustCommand*{\bigcdot}{%
  \mathbin{\mathpalette\bigcdot@{}}%
}
\newcommand*{\bigcdot@scalefactor}{.5}
\newcommand*{\bigcdot@widthfactor}{1.15}
\newcommand*{\bigcdot@}[2]{%
  \sbox0{$#1\vcenter{}$}% math axis
  \sbox2{$#1\cdot\m@th$}%
  \hbox to \bigcdot@widthfactor\wd2{%
    \hfil
    \raise\ht0\hbox{%
      \scalebox{\bigcdot@scalefactor}{%
        \lower\ht0\hbox{$#1\bullet\m@th$}%
      }%
    }%
    \hfil
  }%
}
\makeatother

\tcbset
{colframe=boxcolor,colback=symbols!7!pagebackground,coltext=pageforeground,
fonttitle=\bfseries,nobeforeafter,center title,size=fbox,boxsep=1.5pt,
top=0mm,bottom=0mm,boxsep=0mm,tcbox raise base}

\def\two{{\<generic>\kern0.05em\<genericb>}}
\def\twoI{{\<Ito>\kern0.05em\<Itob>}}

\def\st{\mathsf{fgt}}
\def\mail#1{\burlalt{#1}{mailto:#1}}

\usepackage{thmtools} %Added%

\begin{document}

\def\st{\mathsf{fgt}}
\def\mail#1{\burlalt{#1}{mailto:#1}}

\title{Multi-indices coproducts
from ODEs to singular SPDEs}
\author{Yvain Bruned$^1$, Yingtong Hou$^1$,}
\institute{ 
 IECL (UMR 7502), Université de Lorraine
  \\
Email:\ \begin{minipage}[t]{\linewidth}
\mail{yvain.bruned@univ-lorraine.fr}
\\
	\mail{yingtong.hou@univ-lorraine.fr}
\end{minipage}}

%Added by Foivos - Shuffle Product symbol
\def\dsqcup{\sqcup\mathchoice{\mkern-7mu}{\mkern-7mu}{\mkern-3.2mu}{\mkern-3.8mu}\sqcup}
%Added by Foivos - Shuffle Product symbol

\maketitle

\begin{abstract}
\ \ \ \ In this work, we introduce explicit formulae for the coproducts at play for multi-indices in ODEs and in singular SPDEs.
The two coproducts described correspond to versions of the Butcher-Connes-Kreimer and extraction-contraction coproducts with multi-indices. The main idea is to use the fact that the coproducts are the adjoint of dual products for which one has an explicit simple formula. We are able to derive the explicit formulae via an inner product defined from a symmetry factor easily computable for multi-indices.
\end{abstract}

\setcounter{tocdepth}{1}
\tableofcontents

\setlength{\parskip}{0.5\baselineskip}
\section{ Introduction }

Multi-indices were introduced for scalar valued singular quasi-linear stochastic partial differential equations (SPDEs) in \cite{OSSW}. The idea was to provide a different combinatorial set from the decorated trees used in \cite{BHZ} for coding local expansions of solutions of these singular dynamics. These local expansions are the core of the theory of Regularity Structures invented by Martin Hairer in \cite{reg}. This theory via a black box on decorated trees \cite{reg,BHZ,CH16,BCCH} covers a large class of equations. Multi-indices would replace decorated trees if one wants to index these expansions by the elementary differentials and not by the recentrered iterated stochastic integrals. The main Hopf algebraic structure for multi-indices has been made precise in \cite{LOT} where the authors described the Hopf algebra needed for recentering stochastic iterated integrals. This Hopf algebra is in correspondence with the one on decorated trees in \cite{reg,BHZ} which is connected to the Butcher-Connes-Kreimer Hopf algebra \cite{Butcher72,CK}. Since then, works related to multi-indices have appeared both on the analytical/probabilistic side and the algebraic side in \cite{BK23,LOTT,T,JZ,BL23,GT}.
For comprehensive surveys on multi-indices, interested readers can refer to \cite{LO23,OST}. Let us stress that multi-indices have appeared in different forms in previous works in the litterature, notably in \cite{DL} where the authors characterise the free Novikov algebra with multi-indices. They appear also in the context of numerical analysis in \cite{MV16} where the authors used composition maps for characterising affine equivariant methods (see also \cite{BEH24} where multi-indice B-series are introduced).

Another Hopf algebra crucial for singular SPDE is the extraction-contraction Hopf algebra introduced in \cite{BHZ}.  It is used for renormalising iterated integrals. It is an extension of the Hopf algebra introduced in \cite{CA} inspired from the substitution of B-series in numerical analysis (see \cite{CHV05,CHV07}). At the level of multi-indices, this Hopf algebra has been introduced in \cite{Li23} in the context of Rough Paths for defining translation of rough ODEs (a similar result with decorated trees was previously established in \cite{BCFP}). Such an approach breaks donwn if one does not want to use extended decorations for SPDEs. One can still get renormalised SPDEs in \cite{BL23} with ideas similar to preparation maps coming from \cite{BR18,BB21b}.

The two Hopf algebras mentioned above have been understood on the dual side by introducing the appropriate post-Lie/pre-Lie products and then via variants of the Guin-Oudom procedure (see \cite{GD,Guin1,ELM,LOT,BK23,JZ}), one can construct an explicit associative products. Up to recently, there was no explicit formulae for the coproducts. The first attempt is made in \cite{GMZ24} where the authors derive an explicit formula for recentring coproduct in the context of ODEs. Their main strategy is to compare it with the Butcher-Connes-Kreimer coproducts going back and forth from multi-indices to rooted trees. In the present work, we push forward the programme for finding explicit coproducts formulae. Indeed, we provide formulae for the Butcher-Connes-Kreimer and extraction-contraction type coproducts at play both in the context of ODEs and SPDEs. We can state our main meta theorem as follows:

\begin{theorem} \label{main_theorem} The ODEs and SPDEs multi-indices coproducts can be computed via explicit formulae involving explicit derivations or their adjoints maps, symmetry factors and admissible splittings of (populated) multi-indices.
\end{theorem}

In order to make more explicit our main theorem, we briefly present the definitions of the main objects. We suppose given abstract variables $ (z_k)_{k \in \mathbb{N}} $. A multi-indice is denoted by 
\begin{equs}
	z^{\beta} := \prod_{k \in \mathbb{N}} z_k^{\beta(k)}, \quad \beta : \mathbb{N} \rightarrow \mathbb{N}
\end{equs}
where $ \beta $ has finite support. We define two derivations: One is $ \partial_{z_k} $ that takes derivatives according to the variable $ z_k $, the other one is $D$ which replaces one variable $ z_k$ by $z_{k+1}$. We define an inner product $ \langle \cdot, \cdot \rangle$ on multi-indices and forest of multi-indices (finite collection of multi-indices) from a symmetry factor $S(\cdot)$ given in \eqref{symmetry_factor_2}. Then, we set $ \bar{D} $  to be the adjoint of the map $D$ for the previous inner product. In Theorem~\ref{alternative_LOT}, one has the following explicit expression for the BCK type coproduct
\begin{equs} \label{explicit_coproduct_BCK_2_intro}
	\Delta z^{\beta} &= 
z^{\mathbf{0}} \otimes z^\beta + z^\beta \otimes z^{\mathbf{0}}
\\&+
\sum_{\substack{{\beta  = \beta_1 + \cdots + \beta_n + \hat{\beta}} \\ n \in \mathbb{N}^*}} \frac{1}{S_{\tiny{\text{ext}}}( \tilde{\prod}_{i=1}^n z^{\beta_i})}
\tilde{\prod}_{i=1}^n z^{\beta_i}    \otimes  \bar{D}^{n} z^{\hat{\beta}},		 
\end{equs}
where $   \tilde{\prod}_{i=1}^n z^{\beta_i}$ is a forest of multi-indices and $ S_{\tiny{\text{ext}}}(\cdot) $ is a symmetry factor connected to $S(\cdot)$. Here, $z^{\mathbf{0}} \otimes z^\beta + z^\beta \otimes z^{\mathbf{0}}$ is the primitive part of the coproduct and $\beta  = \beta_1 + \cdots + \beta_n + \hat{\beta} $ is an admissible splitting for this coproduct: The $\beta_i$ and $ \hat{\beta} $ need to satisfy certain combinatorial conditions. For the extraction-contraction coproduct, one needs a different admissible splitting. In the SPDEs context, the reasoning is the same but it is more complex due to the presence of other derivatives.
Let us list our main definitions/results below
\begin{itemize}
	\item Definitions of derivations: \eqref{derivation_ODE}, \eqref{derivations_reg}.
	\item Coproducts with derivations: Propositions~\ref{LOT_coproduct}, \ref{LOT_coproduct2}, \ref{LOT_SPDE}, \ref{extraction_SPDE_1}.
	\item Explicit expression of various adjoint maps: Propositions~\ref{bar_D}, \ref{adjoint_SPDEs} .
	\item Coproducts with adjoint maps: Theorems~\ref{alternative_LOT}, \ref{LOT_coproduct_3}, \ref{LOT_SPDE_A}, \ref{extraction_SPDE_A}.
	\end{itemize}

In comparison to \cite{GMZ24}, we do not use the same symmetry factor but rather the ones which have been historically introduced for multi-indices. Moreover, we avoid the use of trees by providing  proofs that refer only to multi-indices. The main ingredient is the inner product and the fact that the symmetry factor of a multi-indice is easy to compute.
Also, its multiplicative property, namely
\begin{equs}
	S(z^{\beta} z^{\bar{\beta}}) = S(z^{\beta}) S(z^{\bar{\beta}})
	\end{equs}
is extremly helpful as such a property is not true for rooted trees. In the context of ODEs, knownig explicit formulae for the two coproducts allows to compute composition and substitution of multi-indices $B$-series (see \cite{BEH24}). 

For SPDE multi-indices, we have chosen to follow the formalism intoduced in \cite{BD23} which differs from \cite{LOT}. The main difference is that in \cite{LOT}, the authors apply the chain rule and expand the computation with derivatives whereas \cite{BD23} is more compact.
In this part, we define the product $ \star_1 $ as an immediate extension from the product used for ODEs. This formula is new from the current litterature. Also as for ODEs, one can compute compositon and substitution of Regularity structures $B$-series (see \cite{B23}) and explicit formulae are useful in  this context.

Let us outline the paper by summarising the content of its sections. In Section \ref{Sec::2},
we start by recalling multi-indices in the context of ODEs motivated from the elementary differentials that one obtains in dimension one in \eqref{elementary_differential_ODE}. We recall the populated condition \eqref{population_ODE} essential for selecting the meaningful multi-indices which are the one appearing in the Picard iteration of the ODE. One important derivation for the sequel is the derivation $  D$ given in \eqref{derivation_ODE} which gives a Novikov product. We finish the section by introducing some important quantities such as symmetry factors for multi-indices $ \eqref{symmetry_factor_1} $ and for forests of multi-indices  \eqref{symmetry_factor_2}. These symmetry factors help us to define an inner product \eqref{inner_product} that will be extensively used in the sequel.

In Section \ref{Sec::3}, we begin by recalling the definition of the product $ \star_2 $ in  \eqref{definition_star_2} which is built out of the derivation $ D $. The adjoint of this product is the coproduct $ \Delta $ (see \eqref{dual_Delta}). Equipped with these definitions, we are able to give an explicit formula for the map $ \Delta $ in 
Proposition~\ref{LOT_coproduct}. This formula is based on all the ways to split a populated multi-indice into a sum of populated multi-indices and a remainder. The formula heavily depends on the derivation $D$ and the use of the inner product previously introduced. Examples of computations are provided in Example~\ref{example_1}.
We refine this explicit formula by computing the derivation's adjoint map $ \bar{D} $ in Proposition~\ref{bar_D} which leads to the second explicit formula in Theorem~\ref{alternative_LOT}. Then, we introduce the insertion product  $\blacktriangleright$ in \eqref{insertion_product}. It is defined from some explicit derivations $ D $ and $ \partial_{z_k} $. It is extended into a simultaneous insertion $ \star_1 $ in \eqref{multiple_insertion}. In Proposition~\ref{LOT_coproduct2}, we obtain an explicit formula for the adjoint of $\star_1$ denoted by $ \Delta^{\!-} $ that we call the extraction-contraction multi-indices coproduct. By using the previously defined adjoint map $\bar D$, we are able to provide an alternative formula for $ \Delta^{\!-} $ in Theorem \ref{LOT_coproduct_3}.

In Section \ref{Sec::4}, we motivate the introduction of a more general class of multi-indices that are used for describing local expansion of solutions of singular SPDEs. We call them SPDE multi-indices. They contain words whose letters do not commute reflecting non-commutative derivatives (see \eqref{non_commutation}). These multi-indices come with a similar populated condition \eqref{populated_SPDE} and derivations $ D^{(\bn)}, \partial_i $ which do not commute in \eqref{non-commutation_2}. 
We choose a natural basis for these multi-indices by odering the letters (derivatives) and we define the forest of these multi-indices in \eqref{SPDE_forest}. In Proposition~\ref{derivation_basis}, we compute the various derivatives in this basis.
We finish the section by introducing
 symmetry factors for SPDE multi-indices $ \eqref{symmetry_SPDE} $ and for forests of SPDE multi-indices  \eqref{symmetry_SPDE_forest} and an inner product \eqref{inner_product_SPDE}.

In Section \ref{Sec::5}, we follow the same steps as in Section~\ref{Sec::3} for defining the various products and coproducts. These objects can be seen as a generalisation of those used for ODEs. Therefore, we have decided to use the same notation. We start by extending the product $ \star_2 $ in \eqref{definition_star_2_spde}. Then, we derive an explicit formula for $ \Delta $ with the derivatives $ D^{(\bn)} $ and $\partial^k  $ in Proposition~\ref{LOT_SPDE}. With explicit formulae of their adjoints in Proposition~\ref{adjoint_SPDEs}, we propose  in Theorem~\ref{LOT_SPDE_A} another formula for $ \Delta $. For the coproduct $ \Delta^{\!-} $, we proceed the same by generalising $ \star_1  $ in \eqref{insertion_multiple_SPDE} and providing an explicit formula in Proposition~\ref{extraction_SPDE_1}. Alternative formula is given in Theorem~\ref{extraction_SPDE_A} thanks to Proposition~\ref{adjoint_SPDEs}.

\subsection*{Acknowledgements}

{\small
	Y. B. and Y. H. thank Dominique Manchon for communicating his recent result on the Butcher-Connes-Kreimer multi-indices coproduct which was an inspiration for this paper.
  Y. B. gratefully acknowledges funding support from the European Research Council (ERC) through the ERC Starting Grant Low Regularity Dynamics via Decorated Trees (LoRDeT), grant agreement No.\ 101075208. 
} 

\section{ODE multi-indices}
\label{Sec::2}
A general $n$-dimensional ordinary differential equation is in the form of
\begin{equs}
dy=f(y)dt, \quad y(0)=y_0 \in \mathbb{R}^n,
\end{equs}
where $f \in \mathcal{C}^{\infty}$ is a smooth function. Butcher-series (or $B$-series) initially introduced by Butcher \cite{Butcher72} is an expansion of the solution $y$. The main idea of $B$-series is iteratively applying Taylor expansion to the solution $y$. One of the main components of a $B$-series are elementary differentials which are monomials of derivatives in the Taylor coefficients. When one is in dimension one, the elementary differentials have the following general form:
\begin{equs} \label{elementary_differential_ODE}
		F_f[z^\beta](y) = \prod_{k \in \mathbb{N}}\left( f^{(k)}(y)\right)^{\beta(k)} 
\end{equs}
in which $f^{(k)}$ is the $k$-th derivative of $f$ with respect to $y$, and $\beta$ satisfies the \emph{population} condition:
\begin{equation} \label{population_ODE}
[\beta] :=	\sum_{k \in \mathbb{N}} (1 - k)\beta(k)  = 1.
\end{equation}
Hence,
\begin{equs}
	z^{\beta} : = \prod_{k \in \mathbb{N}} z_k^{\beta(k)}
\end{equs}
 with abstract variables $(z_k)_{k \in \mathbb{N}}$ representing $f^{(k)}$ could be used as an appropriate and concise algebraic expression of the solution of an $1$-dimensional ODE
\begin{equation}
	\label{eq:ivp}
	dy=f(y)dt, \quad y(0)=y_0 \in \mathbb{R}.
\end{equation}
We call this combinatorial variable $z^{\beta}$ an ODE multi-indice. Since we are only interested in multi-indices appearing in the expansion of the solution $y$, we would focus only on \emph{populated} ODE multi-indices $z^\beta$ in which $\beta$ satisfies the population condition \eqref{population_ODE}. In the sequel, $\CM$ denotes the set of all populated ODE multi-indices. In the algebra of ODE multi-indices, there is a free Novikov product $ \triangleright $ 
\begin{equs}
	z^{\beta'}\triangleright  z^{\beta} := z^{\beta'}D(z^{\beta}),
\end{equs}
where $D$ is the derivation given by
\begin{equs} \label{derivation_ODE}
	D = \sum_{k \in \mathbb{N}} z_{k+1} \partial_{z_k}.  
\end{equs}
Here, $ \partial_{z_k} $ is the partial derivative in the coordinate $ z_k $. Correspondingly, elementary differentials have the morphism property
\begin{equs}
	F_f[z^{\beta'}\triangleright  z^{\beta}] = F_f[z^{\beta'}]\partial_yF_f[z^{\beta}].
\end{equs}
The product $ \triangleright $ can be extended via the Guin-Oudom procedure \cite{GD}\cite{Guin1} to a product which can be further through the deshuffle corpoduct converted to an associative Grossman-Larson product in the symmetric space over populated multi-indices. This vector space is formed of forests of multi-indices which are collections of multi-indices without any order among them. The forest product $\tilde{\prod}_{i=1}^m z^{\beta_i}$ (or $z^\beta \odot z^\alpha$) is the juxtaposition of multi-indices. The identity element of the forest product is the empty multi-index $z^{\mathbf{0}}$ which is the multi-indice with $\beta(k) =0$ for every $k \in \mathbb{N}$. In the sequel, $\CF$ denotes the set consisting of all forests of multi-indices. 

Norms and symmetry factors are two important measures of multi-indices. We would use the following notations and formulae for them.
\begin{itemize}
	\item The norm of a multi-indice $z^\beta$ is the sum of each element of $\beta$
	\begin{equs}
		|z^\beta| = \sum_{k \in \mathbb{N}} \beta(k).
	\end{equs}
	\item The symmetry factor of a multi-indice counts the total number of possible ways to permute $k$ children of each node $z_k$
	\begin{equs} \label{symmetry_factor_1}
		S(z^{\beta}) = \prod_{k \in \mathbb{N}} (k!)^{\beta(k)}
	\end{equs}
\end{itemize}
The above-mentioned measures can be extended and applied to a forest of multi-indices. Then the norms are
\begin{equs}
	\left|\tilde{\prod}_{i=1}^m z^{\beta_i}\right| = \sum_{i=1}^m |z^{\beta_i}|.
\end{equs}
The symmetry factor for a forest of multi-indices $\tilde{z}^{\tilde{\beta}} =\tilde{\prod}_{i=1}^m \left(z^{\beta_i}\right)^{r_i}$	with disjoint $\beta_i$ is
\begin{equs}  \label{symmetry_factor_2}
	S\left(\tilde{z}^{\tilde{\beta}}\right) = \prod_{i=1}^m r_i!\left(S(z^{\beta_i})\right)^{r_i}.
\end{equs} 
In the sequel, we will also use the following symmetry factor:
\begin{equs}
		S_{\tiny{\text{ext}}}( \tilde{\prod}_{i=1}^n z^{\beta_i}) = 	\frac{	S( \tilde{\prod}_{i=1}^n z^{\beta_i}) }{S( \prod_{i=1}^n z^{\beta_i})}.
\end{equs}
For two forests of multi-indices $\tilde{z}^{\tilde{\alpha}}$ and  $\tilde{z}^{\tilde{\beta}}$, the pairing is defined via the inner product
\begin{equs} \label{inner_product}
	<\tilde{z}^{\tilde{\alpha}}, \tilde{z}^{\tilde{\beta}}> := S(\tilde{z}^{\tilde{\alpha}})\delta^{\tilde{\alpha}}_{\tilde{\beta}},
\end{equs}
where $\delta^{\tilde{\alpha}}_{\tilde{\beta}} = 1$, if ${\tilde{\alpha}}={\tilde{\beta}}$, otherwise it equals $0$.

\section{Explicit formulae for ODEs multi-indices coproducts}
\label{Sec::3}

There are two coproducts in the context of renormalisation through Hopf algebras of rooted trees, namely Butcher-Connes-Kreimer coproduct which is the dual of the Grossman-Larson product and extraction-contraction coproduct which is the dual of the insertion of trees. They play important roles
especially in the composition and substitution laws of $B$-series (see \cite{CHV07} for the rooted tree version and \cite{BEH24} for the multi-indices version). For rooted trees, the Butcher-Connes-Kreimer coproduct was introduced by Connes and Kreimer \cite{CK} and Calaque, Ebrahimi-Fard and Manchon gave the formula for the extraction-contraction coproduct \cite{CA}. Recently, Zhu, Gao and Manchon found the explicit formula of Butcher-Connes-Kreimer type coproduct of multi-indices Hopf algebra \cite{GMZ24}. This coproduct is called Linares–Otto–Tempelmayr coproduct named after the authors who initially introduced multi-indices in the context of singular SPDEs \cite{LOT}. 
In this section we would give an explicit formula of the Linares–Otto–Tempelmayr coproduct which is equivalent to the one in \cite{GMZ24} but derived without rooted trees. The pairing used here is also different from \cite{GMZ24}. The previously unknown explicit formula of the extraction-contraction coproduct for multi-indices will also be introduced.

The Linares–Otto–Tempelmayr coproduct denoted by $ \Delta $ in  \cite{LOT} is the dual of the product $ \star_2$ defined on forests of multi-indices. It is given by
\begin{equs} \label{definition_star_2}
	\tilde{\prod}_{i=1}^n z^{\beta_i} \star_2 z^{\beta} = \prod_{i=1}^n z^{\beta_i} D^n z^{\beta}.
\end{equs}
It is also possible to graft the empty multi-index $z^\mathbf{0}$ to another multi-indice and graft some forest of multi-indices to $z^\mathbf{0}$. We will use the following convention 
\begin{equs} \label{convention}
	z^\mathbf{0} \star_2 z^{\beta} =  z^{\beta}, \text{ and} \quad	
	\tilde{\prod}_{i=1}^n z^{\beta_i} \star_2 z^\mathbf{0} = \tilde{\prod}_{i=1}^n z^{\beta_i}
\end{equs}
Since the  Linares–Otto–Tempelmayr coproduct is the adjoint of the product $\star_2$, by the duality and pairing one has
\begin{equs} \label{dual_Delta}
\langle	\tilde{\prod}_{i=1}^n z^{\beta_i} \star_2 z^{\bar{\beta}}, z^{\beta} \rangle 
= \langle	\tilde{\prod}_{i=1}^n z^{\beta_i} \otimes z^{\bar{\beta}}, \Delta z^{\beta} \rangle. 
\end{equs}
Using the fact that \eqref{definition_star_2} is quite explicit, we want to derive an explicit formula for the coproduct $ \Delta $ through the equality \eqref{dual_Delta}.

\begin{proposition} \label{LOT_coproduct}
	The explicit formula of Linares–Otto–Tempelmayr coproduct is 
	\begin{equs} \label{explicit_coproduct_BCK_1}
		\begin{aligned}
			&\Delta z^{\beta}  = z^{\mathbf{0}} \otimes z^\beta + z^\beta \otimes z^{\mathbf{0}}
			\\& + 
			\sum_{\substack{{\beta  = \beta_1 + \cdots + \beta_n + \hat{\beta}} \\ n \in \mathbb{N}^*}}
			\sum_{z^{\bar\beta} \in \CM}
			\frac{S(z^{\beta})}{	S( \tilde{\prod}_{i=1}^n z^{\beta_i}) S(z^{\bar{\beta}})}  \frac{\langle D^{n} z^{\bar{\beta}}, z^{\hat\beta} \rangle}{S(z^{\hat\beta})}  \tilde{\prod}_{i=1}^n z^{\beta_i}    \otimes  z^{\bar{\beta}}
		\end{aligned}		 
	\end{equs}
	where the $\beta_i$ and $ \bar{\beta} $ satisfy the population condition. 
	The sum over the $\beta_i$ and $  \hat{\beta}$ does not count the order among $\beta_i$, which means $\beta  = \beta_1 + \cdots + \beta_n + \hat{\beta}$ is a partition over $\beta$. 
	%For example, $[0,1]+[1,0]$ and $[1,0]+[0,1]$ are regarded to be the same and only one of them would be count. 
\end{proposition}
\begin{proof}
	There is no order among multi-indices $\beta_i$ ensuring that $\tilde{\prod}_{i=1}^n z^{\beta_i}$ form a forest without repetition.
	Since the primitive parts of the coproduct and the $\sum_{\beta  = \beta_1 + \cdots + \beta_n + \hat{\beta}}$ go through all possible partitions of $\beta$, the form in \eqref{explicit_coproduct_BCK_1} covers all forests $\tilde{\prod}_{i=1}^n z^{\beta_i}$ and populated multi-indices $z^{\bar{\beta}}$ such that their $\star_2$ product has the term formed by $z^\beta$. Therefore, the proof is equivalent to show that each term in \eqref{explicit_coproduct_BCK_1} satisfies the adjoint identity \eqref{dual_Delta}. The equality for primitive element $z^{\mathbf{0}} \otimes z^\beta + z^\beta \otimes z^{\mathbf{0}} $ is immediate from the convention \eqref{convention}.
	We then deal with the non-primitive terms. From left-hand side of \eqref{dual_Delta}, 
	one can observe that the product can be rewritten as a sum over all non-populated multi-indices $z^{\check{\beta}}$:
	\begin{equs}
		\tilde{\prod}_{i=1}^n z^{\beta_i} \star z^{\bar{\beta}}  & = \prod_{i=1}^n z^{\beta_i} D^{n} z^{\bar{\beta}} \\
		& = \sum_{z^\beta \in \CM} \frac{\langle \prod_{i=1}^n z^{\beta_i} D^{n} z^{\bar{\beta}}, z^{{\beta}}  \rangle}{S(z^{{\beta}})} z^{{\beta}}
		\\
		& = \prod_{i=1}^n z^{\beta_i} \sum_{z^{\check{\beta}}}  \frac{\langle D^{n} z^{\bar{\beta}}, z^{\check{\beta}} \rangle}{S(\check{\beta})} z^{\check{\beta}}
	\end{equs} 	
	where from the second line to the third line, we have used the fact that the symmetry factor $S$ is multiplicative in the sense that one has
	\begin{equs}
		S(z^{\beta_i} z^{\bar{\beta}}) = 	S(z^{\beta_i}) S(z^{\bar{\beta}}).
	\end{equs}
	Then, one has
	\begin{equs}
		& \langle \tilde{\prod}_{i=1}^n z^{\beta_i} \star_2 z^{\bar{\beta}} ,	 z^{\beta} \rangle
		= S(z^{\beta})  \frac{\langle D^{n} z^{\bar{\beta}}, z^{\beta-\sum_{i=1}^n \beta_i} \rangle}{S(z^{\beta-\sum_{i=1}^n \beta_i})}.
	\end{equs}
	On the other hand, the coproduct can be viewed as the sum of ``admissible cuts" over the multi-index $z^\beta$
	\begin{equs}
		\Delta z^{\beta} &= 	z^{\mathbf{0}} \otimes z^\beta + z^\beta \otimes z^{\mathbf{0}}
		\\&+
		\sum_{\substack{{\beta  = \beta_1 + \cdots + \beta_n + \hat{\beta}} \\ n \in \mathbb{N}^*}}
		 \sum_{z^{\bar\beta} \in \CM}
		C(\tilde{\prod}_{i=1}^n z^{\beta_i},z^{\bar{\beta}},z^{\beta}) \tilde{\prod}_{i=1}^n z^{\beta_i} \otimes z^{\bar{\beta}}.
	\end{equs}
	Now, we can use duality to compute these coefficients $C$. From
	\begin{equs}
		\langle \tilde{\prod}_{i=1}^n z^{\beta_i} \otimes z^{\bar{\beta}} ,	\Delta z^{\beta} \rangle = S( \tilde{\prod}_{i=1}^n z^{\beta_i}) S(z^{\bar{\beta}})  C(\tilde{\prod}_{i=1}^n z^{\beta_i},z^{\bar{\beta}},z^{\beta}).
	\end{equs}
	one can deduce that
	\begin{equs}
		S( \tilde{\prod}_{i=1}^n z^{\beta_i}) S(z^{\bar{\beta}})  C(\tilde{\prod}_{i=1}^n z^{\beta_i},z^{\bar{\beta}},z^{\beta}) =
		S(z^{\beta})  \frac{\langle D^{n} z^{\bar{\beta}}, z^{\beta-\sum_{i=1}^n \beta_i} \rangle}{S(z^{\beta-\sum_{i=1}^n \beta_i})}
	\end{equs}
	which implies
	\begin{equs}
		C(\tilde{\prod}_{i=1}^n z^{\beta_i},z^{\bar{\beta}},z^{\beta}) & =
		\frac{S(z^{\beta})}{	S( \tilde{\prod}_{i=1}^n z^{\beta_i}) S(z^{\bar{\beta}})}  \frac{\langle D^{n} z^{\bar{\beta}}, z^{\hat\beta} \rangle}{S(z^{\hat\beta})}.
	\end{equs}
	Recall that $ \hat{\beta} = \beta - \sum_{i=1}^n \beta_i$. 
\end{proof}
The following example illustrates how to compute the Linares–Otto–Tempelmayr coproduct $\Delta$.
\begin{example} \label{example_1}
	We consider the multi-indice $z^\beta = z_0^2z_1z_2$. The first step is to partition $z^\beta$ into $n$ populated multi-indices $\beta_i$ and a non-populated multi-indice $\hat\beta$. For this specific $z^\beta$ we have 5 possibilities shown in the following table in which $\odot$ denotes the forest product of multi-indices.
	\begin{table}[H]
		\centering
		\begin{tabular}{||c c c c||} 
			\hline
			Partition $(j)$ & $n$& $\tilde\prod_{i=1}^nz^{\beta_i}$ & $z^{\hat{\beta}_j}$ \\ [0.5ex] 
			\hline\hline
			1 & 1& $z_0$  & $z_0z_1z_2$ \\[1ex] 
			\hline
			2 & 1& $z_0z_1$  & $z_0z_2$ \\ [1ex]
			\hline
			3 & 1& $z_0^2z_2$ & $z_1$ \\ [1ex]
			\hline
			4 & 2& $z_0 \odot z_0$   & $z_1z_2$ \\ [1ex]
			\hline
			5 & 2& $z_0\odot z_0z_1$  & $z_2$ \\ [1ex] 
			\hline
		\end{tabular}
	\end{table}
	Then, we need to find all the $\bar\beta$ such that we can get $\hat\beta$ after applying $D^n$ to $\bar\beta$ and to compute the coefficient $\frac{\langle D^{n} z^{\bar{\beta}}, z^{\hat\beta} \rangle}{S(z^{\hat\beta})}$. For $z^{\hat{\beta}_1}$, we have two possible populated multi-indices $z^{\bar\beta}$ which are $z^{\bar\beta_{1,1}} = z_0^2z_2$ and $z^{\bar\beta_{1,2}} = z_0z_1^2$. Here we take $z^{\bar\beta_{1,1}}$ as an example showing how to compute the coefficient:
	\begin{equs}
		Dz^{\bar\beta_{1,1}} = 2z_0z_1z_2 + z_0^2z_3.
	\end{equs}
	Therefore,
	\begin{equs}
		\frac{\langle D^{n} z^{\bar{\beta}_{1,1}}, z^{\hat\beta_1} \rangle}{S(z^{\hat\beta_1})} = 2.
	\end{equs}
	The computation would be the same for other $z^{\hat\beta_j}$ and the result is shown below.
	\begin{table}[H]
		\centering
		\begin{tabular}{||c c c c c||} 
			\hline
			Partition $(j)$ & $n$& $\tilde\prod_{i=1}^nz^{\beta_i}$ & $z^{\hat{\beta}_j}$& $\frac{\langle D^{n} z^{\bar{\beta}_j}, z^{\hat\beta_j} \rangle}{S(z^{\hat\beta_j})}z^{\bar\beta_j} $ \\ [0.5ex] 
			\hline\hline
			1 & 1& $z_0$  & $z_0z_1z_2$& $2z_0^2z_2$, $2z_0z_1^2$\\[1ex] 
			\hline
			2 & 1& $z_0z_1$  & $z_0z_2$& $z_0z_1$\\ [1ex]
			\hline
			3 & 1& $z_0^2z_2$ & $z_1$& $z_0$\\ [1ex]
			\hline
			4 & 2& $z_0 \odot z_0$   & $z_1z_2$& $3 z_0z_1$\\ [1ex]
			\hline
			5 & 2& $z_0\odot z_0z_1$  & $z_2$& $z_0$\\ [1ex] 
			\hline
		\end{tabular}
	\end{table}
	Finally we have
	\begin{equs}
		\Delta(z^\beta) &= z^{\mathbf{0}} \otimes  z_0^2z_1z_2 +  z_0^2z_1z_2 \otimes z^{\mathbf{0}} + 2z_0 \otimes z_0^2z_2 + 4z_0 \otimes z_0z_1^2
		\\&+ 2 z_0z_1 \otimes z_0z_1 + z_0^2z_2 \otimes z_0 
		+ 3z_0 \odot z_0 \otimes z_0z_1 + 2 z_0\odot z_0z_1 \otimes z_0.
	\end{equs}
\end{example}

In \cite{GMZ24}, the authors obatined a similar explicit formula by guesssing it and comparing it with the explicit expression of the Butcher-Connes-Kreimer coproduct for forests. We took a different path by using the product $ \star_2 $ whose explicit expression is known and it is quite easy to compute via the derivative $D$. In order to be closer to the formula proposed in \cite{GMZ24}, one has to compute the adjoint $ \bar{D} $ of $D$ which is defined by
\begin{equs}
	\langle D z^{\bar{\beta}}, z^{\hat\beta} \rangle = \langle  z^{\bar{\beta}}, \bar{D} z^{\hat\beta} \rangle.
\end{equs}
Its explicit expression is the subject of the next proposition

\begin{proposition} \label{bar_D}
	An explicit expression for the adjoint   $  \bar{D}$ is given by
	\begin{equs} \label{definition_adjoint}
		\bar{D} z^{\hat{\beta}} = \sum_{k \in \mathbb{N}^*} 	 k\frac{\hat{\beta}(k-1)+1}{ \hat{\beta}(k)} 	 z_{k-1} \partial_{z_k} z^{\hat{\beta}}.
	\end{equs}
\end{proposition}
\begin{proof}
	We suppose that $z^{\bar{\beta}}$ and $ z^{\hat{\beta}} $ differ by one $ z_k $ and $ z_{k-1} $:
	\begin{equs}
		\hat{\beta}(k) = \bar{\beta}(k) + 1, \quad 	\hat{\beta}(k-1) = \bar{\beta}(k-1) - 1.
	\end{equs}
	One first has
	\begin{equs}
		\langle D z^{\bar{\beta}}, z^{\hat\beta} \rangle = \bar{\beta}(k-1)S(z^{\hat{\beta}}) = \bar{\beta}(k-1)\frac{k!}{(k-1)!}S(z^{\bar{\beta}}).
	\end{equs}
	Suppose $\bar{D} z^{\hat{\beta}} = \sum_{k \in \mathbb{N}^*} K_{(\hat{\beta},k)}z_{k-1}\partial_{z_k} z^{\hat{\beta}}$ for some constants $K_{\hat{\beta}}$. 
	Then, one has
	\begin{equs}
		\langle  z^{\bar{\beta}}, \bar{D} z^{\hat\beta} \rangle = \hat{\beta}(k)  K_{(\hat{\beta},k)} S(z^{\bar{\beta}}).
	\end{equs}
	Thus,
	\begin{equs}
		K_{(\hat{\beta},k)} = k\frac{\bar{\beta}(k-1)}{\hat{\beta}(k)} = k\frac{\hat{\beta}(k-1)+1}{\hat{\beta}(k)}.
	\end{equs}
\end{proof}
\begin{example}
	We will take $z^{\hat\beta} = z_0z_1^2z_2$ as an example to show how $\bar D$ works.
	\begin{equs}
		\bar D z^{\hat\beta} = (1\frac{1+1}{2})2z_0^2z_1z_2 + (2\frac{2+1}{1})z_0z_1^3 = 2z_0^2z_1z_2 + 6z_0z_1^3
	\end{equs}
	Let $\bar\beta_1 = z_0^2z_1z_2$	and $\bar\beta_2 = z_0z_1^3$. Then, we have 
	\begin{equs}
		\langle z^{\bar\beta_1}, \bar D z^{\hat\beta} \rangle = 4, \quad \langle z^{\bar\beta_2}, \bar D z^{\hat\beta} \rangle = 6.
	\end{equs}
	One can also check that
	\begin{equs}
		\langle D z^{\bar{\beta}_1}, z^{\hat\beta} \rangle = 4, \quad \langle D z^{\bar{\beta}_2}, z^{\hat\beta} \rangle = 6.
	\end{equs}
\end{example}

\begin{theorem} \label{alternative_LOT} An alternative explicit formula to \eqref{explicit_coproduct_BCK_1} of Linares–Otto–Tempelmayr coproduct is 
	\begin{equs} \label{explicit_coproduct_BCK_2}
		\Delta z^{\beta} &= 
		z^{\mathbf{0}} \otimes z^\beta + z^\beta \otimes z^{\mathbf{0}}
		\\&+
		\sum_{\substack{{\beta  = \beta_1 + \cdots + \beta_n + \hat{\beta}} \\ n \in \mathbb{N}^*}} \frac{1}{S_{\tiny{\text{ext}}}( \tilde{\prod}_{i=1}^n z^{\beta_i})}
		\tilde{\prod}_{i=1}^n z^{\beta_i}    \otimes  \bar{D}^{n} z^{\hat{\beta}},		 
	\end{equs}
	where the $\beta_i$ and $ \bar{\beta} $ satisfy the population condition. The sum $\sum_{\beta  = \beta_1 + \cdots + \beta_n + \hat{\beta}}$ does not count the order among $\beta_i$, which means $\beta  = \beta_1 + \cdots + \beta_n + \hat{\beta}$ is a partition over $\beta$.
\end{theorem}
\begin{proof}
	One can move from \eqref{explicit_coproduct_BCK_1} to \eqref{explicit_coproduct_BCK_2} thanks to the following computation:
	\begin{equs}
		\sum_{z^{\bar\beta} \in \CM}
		\frac{S(z^{\beta})}{	S( \tilde{\prod}_{i=1}^n z^{\beta_i}) S(z^{\bar{\beta}})}  \frac{\langle D^{n} z^{\bar{\beta}}, z^{\hat\beta} \rangle}{S(z^{\hat\beta})} & = 	\sum_{z^{\bar\beta} \in \CM}
		\frac{S(z^{\beta})}{	S( \tilde{\prod}_{i=1}^n z^{\beta_i}) S(z^{\hat{\beta}})}  \frac{\langle  z^{\bar{\beta}}, \bar{D}^n z^{\hat\beta} \rangle}{S(z^{\bar\beta})}
		\\ & = \sum_{z^{\bar\beta} \in \CM} \frac{1}{S_{\tiny{\text{ext}}}( \tilde{\prod}_{i=1}^n z^{\beta_i})}
		\frac{\langle  z^{\bar{\beta}}, \bar{D}^n z^{\hat\beta} \rangle}{S(z^{\bar\beta})}.
	\end{equs}
	where we have used the fact that
	\begin{equs}
		S(z^{\beta}) = \prod_{i=1}^n S(z^{\beta_i}) S(z^{\hat{\beta}}).
	\end{equs}
\end{proof}
\begin{example}
	One can repeat example~\ref{example_1} and check that by using this alternative formula the same result as in the first formula will be recovered. We only give the full computation for $z^{\hat{\beta}_1} = z_0z_1z_2$
	\begin{equation*}
		\bar D z^{\hat{\beta}_1} = 1\frac{1+1}{1}z_0^2z_2 + 2 \frac{1+1}{1} z_0z_1^2 = 2z_0^2z_2 + 4z_0z_1^2.
	\end{equation*}
	Other terms are listed in the following table.
	\begin{table}[H]
		\centering
		\begin{tabular}{||c c c c c||} 
			\hline
			Partition $(j)$ & $n$& $\tilde\prod_{i=1}^nz^{\beta_i}$ & $z^{\hat{\beta}_j}$& $\bar{D}^{n} z^{\hat{\beta}} $ \\ [0.5ex] 
			\hline\hline
			1 & 1& $z_0$  & $z_0z_1z_2$& $2z_0^2z_2+4z_0z_1^2$\\[1ex] 
			\hline
			2 & 1& $z_0z_1$  & $z_0z_2$& $2z_0z_1$\\ [1ex]
			\hline
			3 & 1& $z_0^2z_2$ & $z_1$& $z_0$\\ [1ex]
			\hline
			4 & 2& $z_0 \odot z_0$   & $z_1z_2$& $6 z_0z_1$\\ [1ex]
			\hline
			5 & 2& $z_0\odot z_0z_1$  & $z_2$& $2z_0$\\ [1ex] 
			\hline
		\end{tabular}
	\end{table}Finally we have
	\begin{equs}
		\Delta(z^\beta) &= z^{\mathbf{0}} \otimes  z_0^2z_1z_2 +  z_0^2z_1z_2 \otimes z^{\mathbf{0}} + 2z_0 \otimes z_0^2z_2 + 4z_0 \otimes z_0z_1^2
		\\&+ 2 z_0z_1 \otimes z_0z_1 + z_0^2z_2 \otimes z_0 
		+ 3z_0 \odot z_0 \otimes z_0z_1 + 2 z_0\odot z_0z_1 \otimes z_0.
	\end{equs} 	
\end{example}
\begin{remark} 
	One can observe that the coefficients of $ \bar{D} z^{\hat{\beta}} $  depend on $ \hat{\beta} $  which is not the case for the derivation $D$. Therefore, we have proposed the two formulae \eqref{explicit_coproduct_BCK_1} and \eqref{explicit_coproduct_BCK_2}.
\end{remark}

The extraction-contraction coproduct is the dual of the product $\star_1$ understood as `` simultaneous insertion" which is the Guin–Oudom generalisation of the product $ \blacktriangleright $ defined below. For $z^\beta, z^\alpha \in \CM$ 
\begin{equs} \label{insertion_product}
	z^\beta \blacktriangleright z^\alpha : =\sum_{k \in \mathbb{N}}\left(D^{k}z^{\beta}\right)
	\left(\partial_{z_{k}}z^\alpha\right)
\end{equs}
which represents replacing one single variable $z_k$ by $D^{k}z^{\beta}$. Notice that $z^\mathbf{0}$ is not populated which means $z^{\beta}$ is not allowed to be the empty multi-index. A similar product to \eqref{insertion_product} has been introduced in  \cite[Def. 4.3]{Li23} which is inspired from a product on trees in \cite[Sec. 3.4]{BM22}. Since both $z_k$ and $D^{k}$ have the degree of $k$ and both $z^{\alpha}$ and $z^{\beta}$ are populated, $z^\beta \blacktriangleright z^\alpha$ produces a linear combination of populated multi-indices.
Simultaneous inserting $(z^{\beta_i})_{i=1,\ldots,n} \in \CM$ into $z^\alpha \in \CM$ is achieved through the following product
\begin{equs} \label{multiple_insertion}
	\tilde{\prod}_{i=1}^n z^{\beta_i} \star_1 z^\alpha:= \sum_{k_1,...,k_n\in \mathbb{N}}  
	\left(\prod_{i=1}^nD^{k_i}z^{\beta_i}\right)\left[\left(\prod_{i=1}^n\partial_{z_{k_i}}\right)z^\alpha\right] .
\end{equs}
Same as the convention in \cite{CA} for the simultaneous tree insertion, we have to put the restriction that 
\begin{equation}\label{convention2}
	n = |z^\alpha|,
\end{equation}
which means the number of multi-indices in the forest is the same as the number of variables $z_{\cdot}$ in the trunk multi-indice.
We can also extend this simultaneous insertion from $z^\alpha \in \CM$ to a forest of multi-indices through Leibniz rule on $\left(\prod_{j=1}^n\partial_{z_{k_j}}\right)$.

\begin{proposition} \label{LOT_coproduct2}
	The explicit formula of the extraction-contraction multi-indices coproduct $\Delta^{\!-} $ is 
\begin{equs} \label{explicit_formula_extraction_1}
	\Delta^{\! -} z^{\beta}  =&
	\sum_{\tilde{\prod}_{i=1}^n z^{\beta_i} \in \CF}
	\sum_{\substack{k_1,...,k_n\in \mathbb{N}\\ \prod_{i} z_{k_i} \in \CM}}   	
	E(\tilde{\prod}_{i=1}^nz^{\beta_i}, z^\alpha ,z^{\beta})  
	\tilde{\prod}_{i=1}^n z^{\beta_i}
	\otimes   z^\alpha, \quad\quad	
\end{equs}
with
\begin{equs}
	E(\tilde{\prod}_{i=1}^nz^{\beta_i}, z^\alpha ,z^{\beta}) 
	=   
	\sum_{\beta  = \hat{\beta}_1 + \cdots + \hat{\beta}_n} \frac{\alpha!S(z^\beta)}{S( \tilde{\prod}_{i=1}^n z^{\beta_i})S(z^\alpha)}
	\prod_{i=1}^n
	\frac{\langle  
		D^{k_i}z^{\beta_i}, z^{\hat{\beta}_i} \rangle }{S(z^{\hat{\beta}_i})},
\end{equs}
where $z^\alpha := \prod_{i=1}^n z_{k_i}$ and we have order among $\hat{\beta}_1, \cdots , \hat{\beta}_n$. For example if $\beta = [1,1]$, we need to count both $\beta = [1,0]+[0,1]$ and $\beta = [0,1]+[1,0]$. This is because we have the order among $k_i$.
Note that $ n \ge 1$ as we cannot insert the empty multi-indice.	
\end{proposition}
\begin{proof}	
We start with a more general form which covers all $\CF \otimes \CM$.
	\begin{equs} 
		\Delta^{\! -} z^{\beta}  =
		\sum_{\tilde{\prod}_{i=1}^n z^{\beta_i} \in \CF}
		\sum_{k_1,...,k_n\in \mathbb{N}}    \sum_{ z^{\alpha} \in \CM}
		E(\tilde{\prod}_{i=1}^nz^{\beta_i}, z^{\alpha} ,z^{\beta})  
		\tilde{\prod}_{i=1}^n z^{\beta_i}
		\otimes   z^{\alpha}, 
	\end{equs}
	one can regard  $\tilde{\prod}_{i=1}^n z^{\beta_i}$ as the forest simultaneously inserted into $z^{\alpha}$ through $\star_1$ and has
	\begin{equs}
		\langle  \tilde{\prod}_{i=1}^n z^{\beta_i} \otimes z^{\alpha} ,	\Delta^{\!-} z^{\beta} \rangle = S( \tilde{\prod}_{i=1}^n z^{\beta_i}) S(z^{\alpha})  	E(\tilde{\prod}_{i=1}^nz^{\beta_i}, z^{\alpha} ,z^{\beta}) .
	\end{equs}
	On the dual side, one has
	\begin{equs}
		&	\tilde{\prod}_{i=1}^n z^{\beta_i} \star_1 z^{\alpha} = \sum_{z^\beta \in \CM}
		\frac{z^\beta}{S(z^\beta)}
		\left\langle \sum_{k_1,...,k_n\in \mathbb{N}}  
		\left(\prod_{i=1}^nD^{k_i}z^{\beta_i}\right)\left[\left(\prod_{i=1}^n\partial_{z_{k_i}}\right)z^\alpha\right] ,	 z^{\beta} \right\rangle 
		\\& = 
		\sum_{\beta  = \hat{\beta}_1 + \cdots + \hat{\beta}_n + \hat{\alpha}}
		\sum_{k_1,...,k_n\in \mathbb{N}}	
		\prod_{i=1}^n
		\frac{\langle  
			D^{k_i}z^{\beta_i}, z^{\hat{\beta}_i} \rangle }{S(z^{\hat{\beta}_i})}
		\frac{
			\langle \prod_{i=1}^n\partial_{z_{k_i}} z^{\alpha} , z^{\hat{\alpha}} \rangle}
		{S(z^{\hat{\alpha}})}  z^{\hat\alpha}\prod_{i=1}^nz^{\hat\beta_i}.
	\end{equs}
However, as we have the restriction $n = |z^\alpha|$, $z^{\hat\alpha}$ can only be $z^{\mathbf{0}}$. In turn, only if $z^\alpha = \prod_{i=1}^n z_{k_i}$, $\langle \prod_{i=1}^n\partial_{z_{k_i}} z^{\alpha} , z^{\hat{\alpha}} \rangle \ne 0$. Therefore, we have the restriction on $k_i$ that $\prod_{i} z_{k_i} \in \CM$. Then, we have 
\begin{equation*}
\frac{
	\langle \prod_{i=1}^n\partial_{z_{k_i}} z^{\alpha} , z^{\hat{\alpha}} \rangle}
{S(z^{\hat{\alpha}})}  = \frac{
	\langle \prod_{i=1}^n\partial_{z_{k_i}}\prod_{i=1}^n z_{k_i} , z^{\mathbf{0}} \rangle}
{1} = \alpha!
\end{equation*}
	which yields
	\begin{equs}
		&\langle  \tilde{\prod}_{i=1}^n z^{\beta_i} \star_1 z^{\alpha}  ,	z^{\beta} \rangle
		\\=& \sum_{\beta  = \hat{\beta}_1 + \cdots + \hat{\beta}_n}
		\sum_{\substack{k_1,...,k_n\in \mathbb{N}\\ \prod_{i} z_{k_i} \in \CM}}   	
		\alpha!S(z^\beta)
		\prod_{i=1}^n
		\frac{\langle  
			D^{k_i}z^{\beta_i}, z^{\hat{\beta}_i} \rangle }{S(z^{\hat{\beta}_i})},
	\end{equs}
where $z^\alpha = \prod_{i=1}^n z_{k_i}$.
	The equality from the duality
	\begin{equs}
		\langle \tilde{\prod}_{i=1}^n z^{\beta_i}\otimes \prod_{i=1}^n z_{k_i}  ,	\Delta^{\!-} z^{\beta} \rangle & = 
		\langle  \tilde{\prod}_{i=1}^n z^{\beta_i} \star_1 \prod_{i=1}^n z_{k_i}  ,	z^{\beta} \rangle
	\end{equs}
	allows us to conclude.
\end{proof}
The following example demonstrates how to compute the extraction-contraction coproduct
\begin{example}\label{example_EC}
	Suppose we have $z^\beta = z_0^2z_2$. 
\begin{itemize}
\item The first step is to find all possible $k$  corresponding to each $1 \le n \le |z^\beta|$ according to the condition that $z^\alpha = \prod_{i=1}^n z_{k_i}$ is populated. In this particular case $n = 1, 2, 3$. For $n=1$, the populated multi-indice $z^\alpha$ can only be $z_0$. Then $k_1 = 0$. For $n=2$ the populated $z^\alpha$ can only be $z_0z_1$. Then we have $k_1 = 0$ and $k_2 = 1$, or $k_1 = 1$ and $k_2 = 0$. Continuing this procedure, we we can find results shown in the table below.
\begin{table}[H]
	\centering
	\begin{tabular}{||c c c c ||} 
		\hline
		Possibilities& $n$ & $z^{\alpha}$ & $k$ \\ [0.5ex] 
		\hline\hline
		1&$ 1$ & $z_0$ &  $k_1 = 0$ \\[1ex] 
		\hline
		2&$ 2$ & $z_0z_1$ & $k_1 = 0$, $k_2 = 1$ \\ [1ex]
		\hline
		3&$ 2$ & $z_0z_1$ & $k_1 = 1$, $k_2 = 0$ \\ [1ex]
		\hline
		4&$3$ & $z_0^2z_2$& $k_1 = k_2 = 0$, $k_3 = 2$ \\ [1ex]
		\hline
		5&$3$ & $z_0^2z_2$& $k_1 = k_3 = 0$, $k_2 = 2$ \\ [1ex]
		\hline
		6&$3$ & $z_0^2z_2$& $k_2 = k_3 = 0$, $k_1 = 2$ \\ [1ex]
		\hline
		7&$3$ & $z_0z_1^2$& $k_1 = k_2 = 1$, $k_3 = 0$ \\ [1ex]
		\hline
		8&$3$ & $z_0z_1^2$& $k_1 = k_3 = 1$, $k_2 = 0$ \\ [1ex]
		\hline
		9&$3$ & $z_0z_1^2$& $k_2 = k_3 = 1$, $k_1 = 0$ \\ [1ex]
		\hline
	\end{tabular}
\end{table}
 \item Secondly, we need to select those possibilities such that there exist populated multi-indices $z^{\beta_1},...,z^{\beta_n}$ such that $\prod_{i=1}^nD^{k_i}z^{\beta_i} = z^\beta$. Moreover, since there is no order among multi-indices in the forest $\tilde{\prod}_{i=1}^nz^{\beta_i}$, we need to delete the repeated cases (with the same forest and  $k$). For example, in possibility 2 we have $z^{\beta_1} = z_0$ and $z^{\beta_2} = z_0z_1$ with $k_1 = 0$, $k_2 = 1$ while in possibility 3 we have  $z^{\beta_2} = z_0$ and $z^{\beta_1} = z_0z_1$ with $k_2 = 0$, $k_1 = 1$. As the forest $z_0 \odot z_0z_1$ with $k_1 = 0$ and $k_2=1$ covers both of them, they actually have the same pairs $(z^{\beta_i}, k_i)$, i.e.,$\{(z^{\beta_1},0),(z^{\beta_2},1)\} = \{(z^{\beta_2},1),(z^{\beta_1},0)\}$. Thus, according to $\sum_{\tilde{\prod}_{i=1}^n z^{\beta_i} \in \CF}$, we only keep one of them. Possibilities 7, 8, and 9 do not have corresponding $\tilde{\prod}_{i=1}^nz^{\beta_i}$, we therefore, delete them. The possibilities left are listed below.
 \begin{table}[H]
 	\centering
 	\begin{tabular}{||c c c c c c ||} 
 		\hline
 		Possibilities& $n$ & $z^{\alpha}$ & $k$ & $\tilde{\prod}_{i=1}^nz^{\beta_i}$& $E$ \\ [0.5ex] 
 		\hline\hline
 		1&$ 1$ & $z_0$ &  $k_1 = 0$& $z_0^2z_2$ & $1$ \\[1ex] 
 		\hline
 		2&$ 2$ & $z_0z_1$ & $k_1 = 0$, $k_2 = 1$& $z_0 \odot z_0z_1$ &$2$ \\ [1ex]
 		\hline
 		3&$3$ & $z_0^2z_2$& $k_1 = k_2 = 0$, $k_3 = 2$& $z_0 \odot z_0 \odot z_0$ & $\frac{1}{3}$  \\ [1ex]
 		\hline
 		4&$3$ & $z_0^2z_2$& $k_1 = k_3 = 0$, $k_2 = 2$ & $z_0 \odot z_0 \odot z_0$ & $\frac{1}{3}$ \\ [1ex]
 		\hline
 		5&$3$ & $z_0^2z_2$& $k_2 = k_3 = 0$, $k_1 = 2$ & $z_0 \odot z_0 \odot z_0$ & $\frac{1}{3}$ \\ [1ex]
 		\hline
 	\end{tabular}
 \end{table}
Note that, possibilities 3, 4, and 5 are the same forest with different $k$. That is they have different pairs $(z^{\beta_i}, k_i)$. For example, possibility 4 has pairs $\{(z^{\beta_1},0),(z^{\beta_2},0),(z^{\beta_3},2)\}$ but possibility 5 has pairs $\{(z^{\beta_1},0),(z^{\beta_2},2),(z^{\beta_3},0)\}$.
We cannot find one form to cover all of them. Thus, they are not repetitions.  Due to $\sum_{k_1,...,k_n\in \mathbb{N}}$, we need to keep all of them. 
\item
 Finally, we add up all possibilities and obtain
\begin{equation*}
	\Delta^{\! -} z^{\beta}  = z_0^2z_1 \otimes z_0 + 2 z_0 \odot z_0z_1 \otimes z_0z_1 + z_0 \odot z_0 \odot z_0 \otimes z_0^2z_1
\end{equation*}
\end{itemize}
\end{example}
%\begin{proposition} \label{adjoint_partial_zk} 
%For a multi-index $ z^{\hat\beta} $, we define the map $ \bar{\partial}_{z_k} $ 
%	\begin{equs}
%	 \bar{\partial}_{z_{k}} z^{\hat\beta}  = \sum_{k \in \mathbb{N}}
%	\frac{\hat\beta(k) +1}{k!} z_k z^{\hat\beta}.
%	\end{equs}
%Then $ \bar{\partial}_{z_k} $ is the adjoint of $ \partial_{z_k} $ in the sense 
%\begin{equs}
%	\langle \partial_{z_k} z^{{\beta}}, z^{\hat\beta} \rangle = \langle  z^{{\beta}}, \bar{\partial}_{z_k} z^{\hat\beta} \rangle.
%\end{equs}	
%
%	\end{proposition}
%\begin{proof}
%	Let  $  z^{\beta} = z_k z^{\hat{\beta}}$, one has $\beta(k) = \hat\beta(k)+1$ and thus
%	\begin{equs}
%		\langle \partial_{z_k} z^{\beta}, z^{\hat{\beta}} \rangle = (\hat\beta(k) + 1) S(z^{\beta}).
%	\end{equs}
%Then, we conclude from the fact that there exists a constant $ K_{\beta} $ such that
%\begin{equs}
%	\langle z^{\beta} , \bar{\partial}_{z_k} z^{\hat{\beta}} \rangle = K_{\beta} S(z^{\beta}) = K_{\beta} k! S(z^{\hat{\beta}}).
%\end{equs}
%	\end{proof}
It is also possible to use the adjoint map $\bar D$ to rewrite the formula.
\begin{theorem} \label{LOT_coproduct_3} The explicit formula of the extraction-contraction multi-indices coproduct $\Delta^{\!-} $ is 
	\begin{equs} \label{explicit_formula_extraction_2}
		\Delta^{\! -} z^{\beta}  =&
		\sum_{\tilde{\prod}_{i=1}^n z^{\beta_i} \in \CF}
		\sum_{\substack{k_1,...,k_n\in \mathbb{N}\\ \prod_{i} z_{k_i} \in \CM}}   	
		E(\tilde{\prod}_{i=1}^nz^{\beta_i}, z^\alpha ,z^{\beta})  
		\tilde{\prod}_{i=1}^n z^{\beta_i}
		\otimes   z^\alpha, \quad\quad	
	\end{equs}
	with
	\begin{equs}
		E(\tilde{\prod}_{i=1}^nz^{\beta_i}, z^\alpha ,z^{\beta}) 
		=  	\sum_{\beta  = \hat{\beta}_1 + \cdots + \hat{\beta}_n}	  \frac{\alpha!}{S(z^\alpha)S_{\tiny{\text{ext}}}(\tilde{\prod}_{i=1}^n z^{\beta_i})}
		\prod_{i=1}^n
		\frac{\langle  
			z^{\beta_i}, \bar D^{k_i} z^{\hat{\beta}_i} \rangle }{S(z^{{\beta}_i})}
	\end{equs}
where $z^\alpha := \prod_{i=1}^n z_{k_i}$ and we have order among $\hat{\beta}_1, \cdots , \hat{\beta}_n$.
Note that $ n \ge 1$ since inserting the empty multi-indice is forbidden.	
%		\begin{equs} \label{explicit_formula_extraction_2}
%			\Delta^{\! -} z^{\beta}  
%			=
%			\sum_{\beta  = \hat{\beta}_1 + \cdots + \hat{\beta}_n}
%			\sum_{\substack{k_1,...,k_n\in \mathbb{N}\\ \prod_{i} z_{k_i} \in \CM}}   	
%			\tilde{\prod}_{i=1}^n\tilde{D}^{k_i}z^{\hat{\beta}_i}
%			\otimes   \prod_{i} z_{k_i} \quad
%		\end{equs}
%	with
%	\begin{equs}
%		\tilde{\prod}_{i=1}^n\tilde{D}^{k_i}z^{\hat{\beta}_i} := \sum_{z^{\beta_i} \in \CM} \frac{1}{S_{\tiny{\text{ext}}}(\tilde{\prod}_{i=1}^n z^{\beta_i})}
%		\prod_{i=1}^n
%		\frac{\langle  
%			z^{\beta_i}, \bar D^{k_i} z^{\hat{\beta}_i} \rangle }{S(z^{{\beta}_i})}
%		\tilde{\prod}_{i=1}^n z^{\beta_i}
%	\end{equs}
%	where we have order among $\hat{\beta}_1, \cdots , \hat{\beta}_n$.
%The projection $\pi$ is to delete the repetition of the forest $\tilde{\prod}_{i=1}^n \bar{D}
%^{k_i}z^{\hat{\beta}_i}$ due to $\sum_{\beta  = \hat{\beta}_1 + \cdots + \hat{\beta}_n}
%\sum_{\substack{k_1,...,k_n\in \mathbb{N}\\ \prod_{i} z_{k_i} \in \CM}} $ (See the second step of Example~\ref{example_EC} for more explanation). This projection is a result of the fact that we do not have order among forest but we have order among $k_i$.
\end{theorem}
\begin{proof} 
We just pass to the adjoint maps in \eqref{explicit_formula_extraction_1}.
The coefficient $E$ in formula \eqref{explicit_formula_extraction_1} is given by
\begin{equs}
	E(\tilde{\prod}_{i=1}^nz^{\beta_i}, z^\alpha ,z^{\beta}) 
&=   
\sum_{\beta  = \hat{\beta}_1 + \cdots + \hat{\beta}_n} \frac{\alpha!S(z^\beta)}{S( \tilde{\prod}_{i=1}^n z^{\beta_i})S(z^\alpha)}
\prod_{i=1}^n
\frac{\langle  
	D^{k_i}z^{\beta_i}, z^{\hat{\beta}_i} \rangle }{S(z^{\hat{\beta}_i})}
\\&= 
\sum_{\beta  = \hat{\beta}_1 + \cdots + \hat{\beta}_n} \frac{\alpha!\prod_{i=1}^nS(z^{{\beta}_i})}{S( \tilde{\prod}_{i=1}^n z^{\beta_i})S(z^\alpha)}
\prod_{i=1}^n
\frac{\langle  
	z^{\beta_i}, \bar D^{k_i} z^{\hat{\beta}_i} \rangle }{S(z^{{\beta}_i})}
	\\&= 
	\sum_{\beta  = \hat{\beta}_1 + \cdots + \hat{\beta}_n} 
	\frac{\alpha!}{S(z^\alpha)S_{\tiny{\text{ext}}}(\tilde{\prod}_{i=1}^n z^{\beta_i})}
	\prod_{i=1}^n
	\frac{\langle  
		z^{\beta_i}, \bar D^{k_i} z^{\hat{\beta}_i} \rangle }{S(z^{{\beta}_i})},
\end{equs}
where we have used the property of  the symmetry factor that ${\prod}_{i=1}^nS(z^{\hat{\beta}_i}) = S(z^\beta)$.
	\end{proof}
\begin{remark}
	One cannot simplify further \eqref{explicit_formula_extraction_2} by using
		\begin{equs}
	\sum_{z^{\beta_1},\ldots,z^{\beta_n} \in \CM}  
		\prod_{i=1}^n
		\frac{\langle  
			z^{\beta_i}, \bar{D}
			^{k_i}z^{\hat{\beta}_i} \rangle }{S(z^{\beta_i})} =  \bar{D}
		^{k_i}z^{\hat{\beta}_i} 
	\end{equs}
because of the extra coefficient $  	\frac{1}{S_{\tiny{\text{ext}}}( \tilde{\prod}_{i=1}^n z^{\beta_i})  }$.
	\end{remark}
One can repeat Example~\ref{example_EC}. The only difference is using $\bar{D}$ instead of $D$ to compute the coefficient $E$.

\section{SPDE multi-indices}\label{Sec::4}
We are looking at the class of subcritical semi-linear SPDEs of the form
\begin{equation}\label{set01}
\left( \partial_t - 	\mcL \right) u = \sum_{\mfl\in\mfL^-} a^\mfl (\mathbf{u}) \xi_\mfl, \quad (t,x) \in \mathbb{R}_+ \times \mathbb{R}^d,
\end{equation}
where $ \mathcal{L} $ is a differential operator, $ \mfL^- $ is a finite set and the $ \xi_{\mfl} $ are space-time noises (random distributions). We suppose that $ 0 \in \mfL^-$ and that $ \xi_{0} = 1 $.
Moreover, $u$ is a function of $d+1$ variables $t=x_0$, $x_1$, \ldots, $x_d$, and each $a^\mfl(\mathbf{u})$ is a function of $u$ and its iterated partial derivatives.    For $\bn=(n_0,\ldots,n_d)\in\mathbb{N}^{d+1}$, we set
 \[
u^{(\bn)}:= (\prod_{i=0}^d \partial_{x_i}^{n_i} )(u),
 \]
From these variables, one can define commutative derivatives $ \partial_{u^{(\bn)}} $. A first observation is that these new derivatives do not commute with the $ \partial_{x_i} $:
\begin{equs} \label{non_commutation}
  \partial_{u^{(\bn)}}\partial_{x_i}  =\partial_{x_i}\partial_{u^{(\bn)}} +  \partial_{u^{(\bn-e_i)}},
\end{equs}
where $  e_i$ is the standard basis vector of $\mathbb{N}^{d+1}$.
One can actually get an expression for $ \partial_{x_i} $ in terms of the $ \partial_{u^{(\bn)}} $ through the chain rule:
\begin{equs}
\partial_{x_i}=\sum_{\bn\in\mathbb{N}^{d+1}}u^{(\bn+e_i)}\partial_{u^{(\bn)}}.
\end{equs}
In the sequel, we will use the short hand notation:
\begin{equs}
	\partial^k = \prod_{i=0}^n \partial_{x_i}^{k_i}.
\end{equs}
 Then, from these derivatives, one defines elementary differentials that will appear in a local expansion of the solution of \eqref{set01}. They are given by
 \[ \partial^k
\prod_{{\bf{n}}\in \mathbb{N}^{d+1}} \partial_{u^{(\bn)}}^{\alpha(\bn)} 
a^\mfl (\mathbf{u}).
 \]
 where $ k \in \mathbb{N}^{d+1} $ and $ \alpha : \mathbb{N}^{d+1} \rightarrow \mathbb{N}  $ has a finite support. One can observe that we have used a specific basis by ordering the derivatives: We have put all the derivatives $ \partial_{u^{(\bn)}} $ on the right. If the derivatives are not in this order thanks to  \eqref{non_commutation}, we can expressed in the desired form.
 
 This definition motivates the introduction of the following SPDE multi-indices a generalisation of the ODE multi-indices.
 First, we consider two types of letters: $b_i$ representing derivatives $\partial_{x_i}$, $0\le i\le d$ and $ \bn  \in \mathbb{N}^{d+1}$ standing for $\partial_{u^{(\bn)}}$. We introduce an abstract associative algebra $\cA$ generated by all these symbols, and impose the relations 
\begin{equs} \label{relation_words}
b_ib_j=b_jb_i,\quad  \bn\mathbf{m}=\mathbf{m}\bn,\quad   \bn b_i  = b_i \bn + (\bn-e_i).
\end{equs}
Because $\bn-e_i\in\mathbb{N}^{d+1}$, we consider $(\bn-e_i)$ as a letter in the alphabet~$\mathbb{N}^{d+1}$. The fact that the letters $ \bn $ and $ b_i $ do not commute reflects the identity \eqref{non_commutation}.
 We shall use a set of formal variables $\coord=(z_{(\mfl,w)})_{(\mfl,w) \in \mfL^-\times\cA}$, which we consider to be linear in the argument $w$, so that 
  \[
z_{(\mfl,c_1w_1+c_2w_2)}=c_1z_{(\mfl,w_1)}+c_2z_{(\mfl,w_2)}.
 \]
Each such variable $ z_{(\mfl,w)} $ corresponds to $\mathrm{D}^{w} a^\mfl(\mathbf{u}(x))$, where $\mathrm{D}^w$ is obtained from $w$ by replacing $b_i$ with $\partial_{x_i}$ and $\bn$ with $\partial_{u^{(\bn)}}$. Multi-indices $\beta$ over $\coord$  measure the frequency of the variables $  z_{(\mathfrak{l},w)} $, so that we can define SPDE multi-indices by monomials
\begin{equs}
	z^{\beta} : = \prod_{(\mfl,w) \in \mfL^-\times\cA}  z_{(\mfl,w)}^{\beta(\mfl,w)}. 
\end{equs}
 Same as ODE multi-indices, we are interested only in \emph{populated} SPDE multi-indices satisfying the condition
\begin{equs}\label{populated_SPDE}
	\sum_{(\mfl,w)} (1 - \length{w})\beta(\mfl,w)  = 1.
\end{equs}
where  $\length{w}$ is the number of letters $\bn\in\mathbb{N}^{d+1}$ in	$w$. One can observe that we do not take into account the letter $b_i$ as the next Taylor iteration cannot happen at monomials of $x$. From the rooted tree point of view, they correspond to terminal edges, which means nodes whose edges connecting to the trunk are decorated by letters $b$ have no incoming edges. Note that while Relations \eqref{relation_words} imply some dependencies between the variables $z_{(\mfl,w)}$, those relations are homogeneous with respect to the degree $\length{w}$, so the populated condition is well defined, and we can define the vector space $\mathbf{M}_{\mathcal{R}}$ as the span of all populated SPDE multi-indices. 

\begin{example}\label{ex_multi_index}(Correspondance between Elementray Differentials and Multi-indices)
	
	Suppose we have the monomial 
	\begin{equs}
		\left(\partial_{x_2}\partial_{x_5}\partial_{u^{(\bn_1)}}^2\partial_{u^{(\bn_2)}}^3\partial_{u^{(\bn_3)}}   a^\mfl(\mathbf{u})\right)^2
		\left(\partial_{u^{(\bn_1)}}\partial_{u^{(\bn_4)}}^2a^\mfl(\mathbf{u})\right)^3 \left(a^\mfl(\mathbf{u})\right)^{17}.
	\end{equs}
	Then, the associated words $w$ with non-zero frequency are $w_1 = b_2b_5\bn_1\bn_1\bn_2\bn_2\bn_2\bn_3$, $w_2 = \bn_1\bn_4\bn_4$, and the empty word $w_3 = \emptyset$.
	Since the frequency of $w_1$, $w_2$ and $w_3$ are $\beta(\mfl,w_1) = 2$, $\beta(\mfl,w_2) = 3$ and  $\beta(\mfl,w_3) = 17$ respectively, we have the following multi-indice representing this monomial:
	\begin{equs}
		z^\beta &= z_{(\mfl,w_1)}^2  z_{(\mfl,w_2)}^3  z_{(\mfl,w_3)}^{17}
		\\& = z_{(\mfl,b_2b_5\bn_1\bn_1\bn_2\bn_2\bn_2\bn_3)}^2  z_{(\mfl,\bn_1\bn_4\bn_4)}^3 z_{(\mfl,\emptyset)}^{17}.
	\end{equs}
\end{example}

\begin{example} (Population Condition)
	
	Consider again the multi-index $z^\beta$ in Example \ref{ex_multi_index}. 
	\begin{equs}
		&\length{w_1} = w_1[\bn_1]+w_1[\bn_2]+w_1[\bn_3] = 2+3+1 = 6,
		\\&
		\length{w_2} = w_2[\bn_1]+w_2[\bn_4] = 1+2 = 3, \quad \length{w_3} = \length{\emptyset} = 0,
		\\& 	\sum_{(\mfl,w)} (1 - \length{w})\beta(\mfl,w)  = (1-6)2 + (1-3)3 + (1-0)17 = 1.
	\end{equs}
	Therefore, $z^\beta$ is populated.
\end{example}

Let us introduce a family of derivations on the space of SPDE multi-indices; we shall denote them $D^{(\bn)}$, $\bn \in \mathbb{N}^{d+1}$, and $\partial_i$, $0\le i\le d$. The derivations $D^{(\bn)}$ and $\partial_{i}$ are defined by
\begin{equs}
	\label{derivations_reg}
	D^{(\bn)} z_{(\mathfrak{l},w)} = z_{(\mathfrak{l},\bn w)}, \quad \partial_{i} z_{(\mathfrak{l},w)} = z_{(\mathfrak{l},b_iw)}
\end{equs}
 Note that these derivations of course respect the linear relations between the variables $z_{(\mathfrak{l},w)}$ coming from Relations \eqref{relation_words}. Moreover, these relations immediately imply that
\begin{equs} \label{non-commutation_2}
  D^{(\bn)} \partial_i  =\partial_i D^{(\bn)} + D^{(\bn-e_i)}.
\end{equs}
Using the commuting derivations $D^{(\bn)}$, we can still define a family of products $\triangleright_\bn$ on the vector space of all SPDE  multi-indices by setting
\begin{equs}
z^{\gamma} \triangleright_\bn z^{\gamma'} = z^{\gamma} D^{(\bn)} (z^{\gamma'}). 
\end{equs}
In the sequel, we will use the following short hand notation: $ \partial^k = \sum_{i=0}^d \partial_i^{k_i} $. One then defines forests of SPDE multi-indices by terms of the form:
\begin{equs} \label{SPDE_forest}
	\partial^k \tilde{\prod}_{i=1}^n z^{\beta_i} D^{(\bn_i)}
\end{equs}
where the product $ \tilde{\prod} $ is a commutative product. We also denote this product by $ \odot $. One has 
\begin{equs} 
	\partial^k \tilde{\prod}_{i=1}^n z^{\beta_i} D^{(\bn_i)} \, \odot \, \partial^{\bar{k}} \tilde{\prod}_{j=1}^m z^{\bar\beta_j} D^{(\bar{\bn}_j)}  = 	\partial^{k+\bar{k}} \tilde{\prod}_{i=1}^n z^{\beta_i} D^{(\bn_i)} \tilde{\prod}_{j=1}^m z^{\bar\beta_j} D^{(\bar{\bn}_j)}.
\end{equs}
 The difference from ODE multi-indice is that one has to keep track of the decorations and therefore to have the derivatives $ D^{(\bn_i)} $. Such elements make sense when one introduces the product given in \eqref{definition_star_2_spde}. The set of all forest is denoted by $\CF$.

Considering the non-commutative derivatives, to simplify subsequent calculations, we would like to standardise the order of letters in the word $w$ such that derivatives with respect to $x_i$ appear in front of derivatives in the direction $u^{(\mathbf{n})}$. That is, in the sequel, we will consider SPDE multi-indices in the form: 
\begin{equation*}
		z^{\beta} : = \prod_{(\mfl,w) \in \mfL^-\times\cA_c} z_{(\mfl,w)}^{\beta(\mfl,w)}.
\end{equation*}
where elements of $ \CA_{c} $ are of the form $ u w $ where $ u \in M(\mathbb{B}) $ with $ \mathbb{B} =  \lbrace b_{0},...,b_{d} \rbrace  $ and $ v = M(\mathbb{N}^{d+1})$. Here for a set $A$, $M(A)$ is the commutative monoid in the letters of $A$. We shall also use notations $ \coord_c=(z_{(\mfl,w)})_{(\mfl,w) \in \mfL^-\times\cA_c} $, $\mathcal{M}_{\mathcal{R}_c}$ denoting the set of all populated SPDE multi-indices with the above-mentioned standardised letter order in $w$, and $\mathbf{M}_{\mathcal{R}_c}$ being the linear span of multi-indices in $\mathcal{M}_{\mathcal{R}_c}$. 
In the sequel, we will make the natural identification of an element of $ M(\mathbb{B}) $ with $ \mathbb{N}^{d+1} $. For every $ u  \in M(\mathbb{B})  $, we associate:
\begin{equs}
	u:= \prod_{i=0}^d b_i^{u[b_i]} \equiv \sum_{i=0}^d u[b_i] e_i,
\end{equs} 
where $u[b_i]$ is the number of letter $b_i$ in the word $u$.
In the basis $ \mathcal{} $, one has an explicit expression for the derivation $ D^{(\bn)} $ and $ \partial_i $ in the next proposition:
\begin{proposition}\label{derivation_basis}
	For every $ (\mfl,w) \in \mfL^- \times \cA_c $ with $ w = uv $,  $ u \in M(\mathbb{B})$ and $ v = M(\mathbb{N}^{d+1})$,  one has:
	\begin{equs} \label{identity_derivation}
		\begin{aligned}
			&\partial_i z_{(\mfl,w)}  = z_{(\mfl,b_i w)},
		\\
		&D^{(\bn)} z_{(\mfl,w)}  = z_{(\mfl,\bn w)} = \sum_{\ell \in \mathbb{N}^{d+1}} D^{(\bn)}_{\ell}  z_{(\mfl,w)},
		\end{aligned}
	\end{equs}
where $ D^{(\bn)}_{\ell} $ is given by
\begin{equs}
D^{(\bn)}_{\ell}  z_{(\mfl,w)} := 	 { u \choose \ell} z_{(\mfl, (u - \ell ) (\bn-\ell) v  )},
\end{equs}
where recall that, we have identified $u$ with $\sum_{i=0}^d u[b_i]e_i$. Thus, $u$ is also a multi-indice whose $i$-th entry is $u[b_i]$, and $u - \ell$ represents the word $\prod_{i=0}^db_i^{(u[b_i]-\ell_i)}$.
	\end{proposition}
\begin{proof}
One first notices that
\begin{equation*}
	z_{(\mfl,w)} = \prod_{i=0}^d \partial_i^{u[b_i]}
	\prod_{\bn \in \mathbb{N}^{d+1}}  (D^{(\bn)})^{v[\bn]} z_{(\mfl,0)}.
\end{equation*}
Then, it boils down to show that
\begin{equation*}
	D^{(\mathbf{m})} \partial^{u}
 = 	\sum_{\ell \in \mathbb{N}^{d+1}} { u \choose \ell} 	 \partial^{u-\ell}
	 D^{(\mathbf{m}-\ell)}.
\end{equation*}
	When one has $ u=b_i $,
		this is a consequence of the fact that
		\begin{equation*}
			D^{(\mathbf{m})} \partial^{b_i} = \partial^{b_i} D^{(\mathbf{m})} + D^{(\mathbf{m}-e_i)}.
		\end{equation*}
	We proceed by recurrence and we suppose it true for $u \in \mathbb{R}^{d+1}$. One has 
	\begin{equs}
		D^{(\mathbf{m})} \partial^{u + b_i}
		 & =  \sum_{\ell \in \mathbb{N}^{d+1}} { u \choose \ell} 	 \partial^{u-\ell}
		D^{(\mathbf{m}-\ell)} \partial^{b_i}
		\\
		&= \sum_{r \in \mathbb{N}^{d+1}} \sum_{\ell \in \mathbb{N}^{d+1}} { u \choose \ell} { e_i \choose r}	 \partial^{(u-\ell)+(b_i-r)}
		D^{(\mathbf{m}-\ell-r)} 
		\\
		& =  \sum_{k \in \mathbb{N}^{d+1}}{ u+e_i \choose k}	 \partial^{u+b_i-k}
		D^{(\mathbf{m}-k)}
	\end{equs}
		where we have used the Chu-Vandermonde identity that tells us that
		\begin{equation*}
			\sum_{\ell+r = k}  {u \choose \ell} {e_i \choose r} =   {u + b_i \choose k}.
		\end{equation*}
	\end{proof}

Similar to the ODE multi-indices, the symmetry factor for SPDE multi-indices is also necessary for defining the pairing in duality.
We define the symmetry factors as the following:
 $ S(z^\beta) $ is the symmetry factor of the multi-indice  $ z^\beta$ given by
\begin{equs} \label{symmetry_SPDE}
	S(z^\beta) := \prod_{(\mfl,w) \in \mfL^- \times \cA_c}  \left(w ! \right)^{\beta(\mfl,w)}.
\end{equs}
For $w = b^k \prod_{i=1}^n {\bf n}_i^{\beta_i}$ where the $ {\bf n}_i $ are disjoints, one has
\begin{equs}
	w !  = k! \prod_{i=1}^n \beta_i !.
	\end{equs}
Here $ b^{k} $ is a short hand notation for $ \prod_i b_i^{k_i} $ and we use the convention that $k! := \prod_i k_i!$. 
It can be extended to the forest of multi-indices by
\begin{equs} \label{symmetry_SPDE_forest}
	S\left(\partial^k  \, \tilde{\prod}_{j=1}^n\left(z^{\beta_j}  D^{(\bn_j)}\right)^{r_j}\right) :=k! \prod_{j=1}^nr_j!S\left(z^{\beta_j}\right)^{r_j},
\end{equs}
where $z^{\beta_j}$ are disjoint. Accordingly, the pairing of (forests of) SPDE multi-indices is defined as
\begin{equs} \label{inner_product_SPDE}
	\langle z^{\beta_1}, z^{\beta_2} \rangle = S(z^{\beta_1})\delta^{z^{\beta_1}}_{z^{\beta_2}}.
\end{equs}

\section{Explicit formulae for SPDE multi-indices coproducts}\label{Sec::5}
Similar to ODEs, there are two coproducts in SPDEs, which were initially introduced in \cite{BHZ} using decorated rooted trees. They can be viewed as the Regularity Structures analogues of Butcher-Connes-Kreimer coproduct and extraction-contraction coproduct. 
We propose in this section explicit formulae for these coproducts in the context of multi-indices.  

	From this section on we do not need the coproducts or products in the ODE version. Therefore,for simplicity in notations, we will use again $\Delta$, $\star_2$, $\Delta^{\! -}$, and $\star_1$. However, they are all in the SPDE cases.
	
The Butcher-Connes-Kreimer type coproduct denoted $ \Delta $ in  \cite{LOT} is the dual to of the associative product $ \star_2$ defined on forest of multi-indices. It is given by
\begin{equs} \label{definition_star_2_spde}
\partial^k	\tilde{\prod}_{i=1}^r z^{\beta_i} D^{(\bn_i)} \star_2 z^{\beta} = \prod_{i=1}^n z^{\beta_i} \partial^k \prod_{i=1}^r D^{(\bn_i)} z^{\beta}.
\end{equs}
We have the same convention as the ODE version that 
\begin{equs} \label{convention_SPDE}
	z^\mathbf{0} \star_2 z^{\beta} =  z^{\beta}, \text{ and} \quad	
	\partial^k	\tilde{\prod}_{i=1}^r z^{\beta_i} D^{(\bn_i)} \star_2 z^\mathbf{0} = \partial^k	\tilde{\prod}_{i=1}^r z^{\beta_i} D^{(\bn_i)}
\end{equs}
By the duality and pairing, one has
\begin{equs} \label{dual_Delta_spde}
	\langle	\partial^k	\tilde{\prod}_{i=1}^r z^{\beta_i} D^{(\bn_i)}  \star_2 z^{\bar{\beta}}, z^{\beta} \rangle 
	= \langle \partial^k	\tilde{\prod}_{i=1}^r z^{\beta_i} D^{(\bn_i)} \otimes z^{\bar{\beta}}, \Delta z^{\beta} \rangle. 
\end{equs}

One can notice that by going to the dual, one has infinite sum for $ \Delta $. This is due to the defintion of the operator $ D^{(\bn_i)} $ that decreases the decoration on a node (see Proposition~\ref{derivation_basis}). The same issue has been observed in \cite{BHZ} and was solved by using a bigrading (see \cite[Sec; 2.3]{BHZ}). One can do the same and propose a bigrading
whose first component will be the sum over the norms of the letter $ \bn $ that appear in $z^{\beta}$ and the second component is the number of letters $ \bn $ plus the number of nodes.

Using the fact that \eqref{dual_Delta_spde} is quite explicit, we want to derive an explicit formula for the coproduct $ \Delta $ through the equality \eqref{dual_Delta_spde}.

\begin{proposition} \label{LOT_SPDE}
	The explicit formula of the Linares–Otto–Tempelmayr coproduct in the context of SPDEs is
		\begin{equs} \label{SPDE_coproduct_BCK_1}
		\begin{aligned}
			\Delta z^{\beta}  =& z^{\mathbf{0}} \otimes z^\beta + z^\beta \otimes z^{\mathbf{0}} +
			\sum_{\partial^k \tilde{\prod}_{i=1}^r z^{\beta_i} D^{(\bn_i)} \in \CF }
			\sum_{z^{\bar\beta} \in \mathcal{M}_{\mathcal{R}_c}}
			\\ &
			\frac{S(z^{\beta})}
			{S( \partial^k \tilde{\prod}_{i=1}^r z^{\beta_i} D^{(\bn_i)} ) S(z^{\bar{\beta}})}
			  \frac{\langle \partial^k \prod_{i=1}^r D^{(\bn_i)} z^{\bar{\beta}}, z^{\hat\beta} \rangle}{S(z^{\hat\beta})}  \partial^k \tilde{\prod}_{i=1}^r z^{\beta_i} D^{(\bn_i)}   \otimes  z^{\bar{\beta}},
		\end{aligned}		 
	\end{equs}
	where $\hat{\beta}  = \beta -\sum_{i=1}^r \beta_i$. 
\end{proposition}
\begin{proof}
First of all, one can check that $	\sum_{\partial^k \tilde{\prod}_{i=1}^r z^{\beta_i} D^{(\bn_i)} \in \CF } \sum_{z^{\bar\beta} \in \mathcal{M}_{\mathcal{R}_c}}$ together with the primitive elements cover all the forest and trunk such that their product $\star_2$ contains the $z^\beta$ term. The adjoint relationship of the primitive elements is trivial.
Therefore, we only deal with other terms and start with the left-hand side of \eqref{dual_Delta_spde}. 
One can observe that the product can be rewritten as a sum over all non-populated standardised SPDE multi-indices $z^{\check{\beta}}$:
\begin{equs}
	\partial^k	\tilde{\prod}_{i=1}^r z^{\beta_i} D^{(\bn_i)}  \star_2 z^{\bar{\beta}}  & = \prod_{i=1}^n z^{\beta_i} \partial^k \prod_{i=1}^r D^{(\bn_i)} z^{\bar\beta} \\
	& = \sum_{z^\beta \in \mathcal{M}_{\mathcal{R}_c}} \frac{\langle \prod_{i=1}^n z^{\beta_i} \partial^k \prod_{i=1}^r D^{(\bn_i)} z^{\bar\beta}, z^{{\beta}}  \rangle}{S(z^{{\beta}})} z^{{\beta}}
	\\
	& = \prod_{i=1}^n z^{\beta_i} \sum_{z^{\check{\beta}}}  \frac{\langle \partial^k \prod_{i=1}^r D^{(\bn_i)} z^{\bar\beta}, z^{\check{\beta}} \rangle}{S(\check{\beta})} z^{\check{\beta}}
\end{equs} 	
where from the second line to the third line, we have used the fact that the symmetry factor $S$ is multiplicative in the sense that one has
\begin{equs}
	S(z^{\beta_i} z^{\bar{\beta}}) = 	S(z^{\beta_i}) S(z^{\bar{\beta}}).
\end{equs}
Then, one has
\begin{equs}
	& \langle \partial^k	\tilde{\prod}_{i=1}^r z^{\beta_i} D^{(\bn_i)}  \star_2 z^{\bar{\beta}} ,	 z^{\beta} \rangle
	= S(z^{\beta})  \frac{\langle \partial^k \prod_{i=1}^r D^{(\bn_i)} z^{\bar\beta}, z^{\beta-\sum_{i=1}^r \beta_i} \rangle}{S(z^{\beta-\sum_{i=1}^r \beta_i})}.
\end{equs}
On the other hand, the coproduct can be viewed as the sum of ``admissible cuts" over the multi-index $z^\beta$
\begin{equs}
	\Delta z^{\beta} 
	=&
	z^{\mathbf{0}} \otimes z^\beta + z^\beta \otimes z^{\mathbf{0}}+
	\sum_{\partial^k \tilde{\prod}_{i=1}^r z^{\beta_i} D^{(\bn_i)} \in \CF }
	 \sum_{z^{\bar\beta} \in \CM_{\mathcal{R}_c}}
	\\&
	C\left( \partial^k	\tilde{\prod}_{i=1}^r z^{\beta_i} D^{(\bn_i)},z^{\bar{\beta}},z^{\beta} \right)
	  \partial^k	\tilde{\prod}_{i=1}^r z^{\beta_i} D^{(\bn_i)} \otimes z^{\bar{\beta}}.
\end{equs}
Now, we can use duality to compute these coefficients $C$. From
\begin{equs}
	&\langle \partial^k	\tilde{\prod}_{i=1}^r z^{\beta_i} D^{(\bn_i)} \otimes z^{\bar{\beta}} ,	\Delta z^{\beta} \rangle 
	\\=& S( \partial^k	\tilde{\prod}_{i=1}^r z^{\beta_i} D^{(\bn_i)}) S(z^{\bar{\beta}})  C\left( \partial^k	\tilde{\prod}_{i=1}^r z^{\beta_i} D^{(\bn_i)},z^{\bar{\beta}},z^{\beta} \right).
\end{equs}
one can deduce that
\begin{equs}
	& S( \partial^k	\tilde{\prod}_{i=1}^r z^{\beta_i} D^{(\bn_i)}) S(z^{\bar{\beta}})  C\left( \partial^k	\tilde{\prod}_{i=1}^r z^{\beta_i} D^{(\bn_i)},z^{\bar{\beta}},z^{\beta} \right)
	\\=&
	S(z^{\beta})  \frac{\langle \partial^k \prod_{i=1}^r D^{(\bn_i)} z^{\bar\beta}, z^{\beta-\sum_{i=1}^r \beta_i} \rangle}{S(z^{\beta-\sum_{i=1}^r \beta_i})}
\end{equs}
which implies
\begin{equs}
C\left( \partial^k	\tilde{\prod}_{i=1}^r z^{\beta_i} D^{(\bn_i)},z^{\bar{\beta}},z^{\beta} \right) & =
	\frac{S(z^{\beta})}{	S( \partial^k	\tilde{\prod}_{i=1}^r z^{\beta_i} D^{(\bn_i)}) S(z^{\bar{\beta}})}  \frac{\langle \partial^k \prod_{i=1}^r D^{(\bn_i)} z^{\bar\beta}, z^{\hat\beta} \rangle}{S(z^{\hat\beta})}.
\end{equs}
Recall that $ \hat{\beta} = \beta - \sum_{i=1}^r \beta_i$.
\end{proof}

	We define the adjoint maps $ D^{(\bn)} $ and $ \partial^k $ denoted by $ \bar{D}^{(\bn)}  $ and $ \bar{\partial}^k $ as
	\begin{equs} \label{adjoint_SPDE1}
		\langle	\prod_{i=1}^rD^{(\bn_i)} z^{\bar\beta}, z^{\hat{\beta}} \rangle  = 	\langle	 z^{\bar\beta}, \prod_{i=1}^r\bar{D}^{(\bn_i)}  z^{\hat{\beta}} \rangle, \quad \langle	\partial^k  z^{\bar\beta}, z^{{\beta}} \rangle  = 	\langle	 z^{\bar\beta},  \bar{\partial}^k z^{{\beta}} \rangle
	\end{equs}
	where $\beta$ and $ \bar{\beta} $ are populated multi-indices and $ \hat{\beta} $ is such that $[\hat\beta] = r+1$. 
	\begin{proposition} \label{adjoint_SPDEs}
		One has the following explicit expression for the coefficents describing $ \bar{\partial}^k$ and $ \bar{D}^{(\bn)} $, where $k \in \mathbb{N}^{d+1}$.
		\begin{equs}
			\bar{\partial}^k = \prod_{i=0}^d \bar{\partial}^{k_i}_i = \prod_{i=0}^d\left(\sum_{(\mfl, uv)
				\in \mfL^- \times \cA_c}
			u[b_i]\frac{(\hat\beta(\mfl,(u-e_i)v)+1)}{\hat\beta(\mfl,uv)} z_{(\mfl,(u-e_i)v)}\partial_{z_{(\mfl,uv)}}
			\right)^{k_i},
		\end{equs}
		\begin{equs}
			\bar{D}^{(\bn)}=& 
			\sum_{\ell \in \mathbb{N}^{d+1}}
			\sum_{(\mfl, (u - \ell ) (\bn-\ell) v  )
				\in \mfL^- \times \cA_c}
%			\\& 
			\frac{(v[\bn-\ell]+1)(\hat\beta(\mfl,uv)+1)}{\ell!\hat\beta(\mfl, (u - \ell ) (\bn-\ell) v  )} z_{(\mfl,uv)}\partial_{z_{(\mfl, (u - \ell ) (\bn-\ell) v  )}}
		\end{equs}
		where $\beta$ is a populated multi-indice and $ \hat{\beta} $ is such that $[\beta] = n+1$.
	\end{proposition}
	\begin{proof}
	(1)	Suppose that $z^{\bar\beta}$ is obtained by removing one letter $b_i$ from a variable $z_{(\mfl,uv)}$ in $z^{\hat\beta}$, which means
	\begin{equs}
	\hat\beta(\mfl,uv) = \bar\beta(\mfl,uv) + 1, \quad \hat\beta(\mfl,(u-e_i)v) = \bar\beta(\mfl,(u-e_i)v) - 1.
	\end{equs}
	Then,
	\begin{equs}
	\langle \partial_i z^{\bar\beta}, z^{\hat\beta} \rangle &= \bar\beta(\mfl,(u-e_i)v) S(z^{\hat\beta}) 
	\\&= \bar\beta(\mfl,(u-e_i)v)\frac{u[b_i]!}{(u[b_i]-1)!}S(z^{\bar\beta})
	\\&=
	(\hat\beta(\mfl,(u-e_i)v)+1)u[b_i]S(z^{\bar\beta})
	.
	\end{equs}
	As for the other inner product, suppose we have for some constants $K$,
	\begin{equs}
	\bar\partial_i z^{\hat\beta} =  \sum_{(\mfl, uv)
		\in \mfL^- \times \cA_c} K_{(\mfl, uv)} z_{(\mfl,(u-e_i)v)}\partial_{z_{(\mfl,uv)}} z^{\hat\beta}
	\end{equs}
	which indicates
	\begin{equs}
	\langle	z^{\bar\beta},\bar\partial_i z^{\hat\beta} \rangle = \hat\beta(\mfl,uv) K_{(\mfl, uv)} S(z^{\bar\beta})
	\end{equs}
	Therefore, by the equality, 
	\begin{equs}
	K_{(\mfl, uv)} = u[b_i]\frac{(\hat\beta(\mfl,(u-e_i)v)+1)}{\hat\beta(\mfl,uv)}
	\end{equs}
	Finally, through induction and by the Leibniz rule, we can conclude
	\begin{equs}
		\langle \partial^k z^{\bar\beta}, z^{\hat\beta} \rangle = \langle	z^{\bar\beta},\bar\partial^k z^{\hat\beta} \rangle.
	\end{equs} 
		
	(2)	Suppose that $z^{\hat\beta}$ is obtained by changing one variable $z_{(\mfl,w)}$ in $z^{\bar\beta}$ to $z_{(\mfl, (u - \ell ) (\bn-\ell) v  )}$ with $w = uv$ and $\ell \in \mathbb{N}^{d+1}$, which means
		\begin{equs}
			\hat{\beta}(\hat{\mfl},\hat{w}) = 
			\begin{matrix}	
				\begin{cases}
					\bar\beta(\hat{\mfl},\hat{w}), \quad \text{ if } (\hat{\mfl},\hat{w}) \ne (\mfl,w) \text{ and } (\hat\mfl,\hat{w} ) \ne (\mfl, (u - \ell ) (\bn-\ell) v  )\\			
					\bar\beta(\mfl,w)-1, \quad \text{ if } (\hat{\mfl},\hat{w}) = (\mfl,w) \\
					\bar\beta(\mfl, (u - \ell ) (\bn-\ell) v  )+1, \quad \text{ if } (\hat{\mfl},\hat{w}) = (\mfl, (u - \ell ) (\bn-\ell) v  ) \\	
				\end{cases}
			\end{matrix}.
		\end{equs}
		Then,
		\begin{equs}
			\langle D^{(\bn)}z^{\bar\beta}, z^{\hat\beta} \rangle &=  { u \choose \ell} \bar\beta(\mfl,w) S(z^{\hat\beta})
			\\&= { u \choose \ell} \bar\beta(\mfl,w) \frac{(u-\ell)!(v[\bn-\ell]+1)!}{u!v[\bn-\ell]!}S(z^{\bar\beta})
			\\&= { u \choose \ell} \bar\beta(\mfl,w) \frac{(u-\ell)!(v[\bn-\ell]+1)}{u!}S(z^{\bar\beta})
		\end{equs}
		Meanwhile, the inner product in dual side yields
		\begin{equs}
			\langle	 z^{\bar\beta}, \bar{D}^{(\bn)}  z^{\hat{\beta}} \rangle  = K \hat\beta(\mfl, (u - \ell ) (\bn-\ell) v  ) S(z^{\bar\beta}),
		\end{equs}
		for some constant K.
		The equality $\langle	D^{(\bn)} z^{\bar\beta}, z^{\hat{\beta}} \rangle  = 	\langle	 z^{\bar\beta}, \bar{D}^{(\bn)}  z^{\hat{\beta}} \rangle $ tells us that 
		\begin{equs}
			K &= { u \choose \ell} \frac{(v[\bn-\ell]+1)(u-\ell)!\bar\beta(\mfl,w)}{u!\hat\beta(\mfl, (u - \ell ) (\bn-\ell) v  )} 
			\\&=
			{ u \choose \ell} \frac{(v[\bn-\ell]+1)(u-\ell)!(\hat\beta(\mfl,w)+1)}{u!\hat\beta(\mfl, (u - \ell ) (\bn-\ell) v  )}
			\\&=
			\frac{(v[\bn-\ell]+1)(\hat\beta(\mfl,w)+1)}{\ell!\hat\beta(\mfl, (u - \ell ) (\bn-\ell) v  )}
			.
		\end{equs}
	\end{proof}
By using adjoint maps, one can simplify the formula in the following way.
\begin{theorem} \label{LOT_SPDE_A}
Another explicit formula of the Linares–Otto–Tempelmayr coproduct in the context of SPDEs is
	\begin{equs} \label{SPDE_coproduct_BCK_2}
			&\Delta z^{\beta}  = z^{\mathbf{0}} \otimes z^\beta 
			+ z^\beta \otimes z^{\mathbf{0}}
			\\ &
			+
			\sum_{\partial^k \tilde{\prod}_{i=1}^r z^{\beta_i} D^{(\bn_i)} \in \CF }		
			\frac{1}{	S_{\tiny{\text{ext}}}( \partial^k \tilde{\prod}_{i=1}^r z^{\beta_i} D^{(\bn_i)} ) } 
			   \partial^k	\tilde{\prod}_{i=1}^r z^{\beta_i} D^{(\bn_i )}   \otimes  \prod_{i=1}^r\bar{D}^{(\bn_i)} \bar{\partial}^k   z^{\hat{\beta}}	 
	\end{equs}
	where $\hat{\beta}  = \beta -\sum_{i=1}^r \beta_i$.  
\begin{proof}
	The proof is the same as the ODE version in Theorem~\ref{alternative_LOT}. We apply the adjoint relation \eqref{adjoint_SPDE1} to the formula in Proposition~\ref{adjoint_SPDEs}.
\end{proof}
\end{theorem}
The following example illustrate the computation of SPDE Linares–Otto–Tempelmayr coproduct.
\begin{example}\label{SPDE_1}
	Suppose there are only two variables $x_0 = t$ and $x_1$ in the SPDE. We also assume that every variable is decorated with the same noise, therefore we omit the decoration $\mfl$ in the following calculations.
	The target is $z^\beta = z_{w_1}^2z_{w_2}z_{w_3}$ with $w_1 = b_0$, $w_2 =  \bn$, and $w_3 = b_1\mathbf{m}\mathbf{m}$, where $\bn = [1,0]$ and $\mathbf{m} = [0,1]$.
\begin{itemize}
	\item Firstly, we need to find all possible forest and use the partition $\beta  = \beta_1 + \cdots + \beta_r + \hat{\beta}$ to find $z^{\hat{\beta}}$.
		\begin{table}[H]
		\centering
		\begin{tabular}{||c c c c||} 
			\hline
			Partition $(j)$ & $r$& $\tilde\prod_{i=1}^rz^{\beta_i}$ & $z^{\hat{\beta}_j}$ \\ [0.5ex] 
			\hline\hline
			1 & 1& $z_{w_1}$  & $z_{w_1}z_{w_2}z_{w_3}$ \\[1ex] 
			\hline
			2 & 1& $z_{w_1}z_{w_2}$  & $z_{w_1}z_{w_3}$ \\ [1ex]
			\hline
			3 & 1& $z_{w_1}^2z_{w_3}$ & $z_{w_2}$ \\ [1ex]
			\hline
			4 & 2& $z_{w_1} \odot z_{w_1}$   & $z_{w_2}z_{w_3}$ \\ [1ex]
			\hline
			5 & 2& $z_{w_1}\odot z_{w_1}z_{w_2}$  & $z_{w_3}$ \\ [1ex] 
			\hline
		\end{tabular}
	\end{table}
	\item The second step is to identify $\bn_i$ (extra letter) such that $\prod_{i=1}^r\bar{D}^{(\bn_i)} \bar{\partial}^k   z^{\hat{\beta}} \ne 0$. Recall that $\bar{D}^{(\bn_i)}$ deletes one letter $\bn_i-\ell$ in a word. For example, in Partition $1$, we need to find one letter $\bn_1-\ell$ to be deleted since $r=1$.
	As $w_1$ has no letters in $v$, possibilities are deleting one letter from $w_2$ or deleting one letter from $w_3$. Therefore, possible $\bn_1-\ell$ are $\bn$ and $\mathbf{m}$. The table below shows all possible $\bn_i$ for each partition.
	\begin{table}[H]
		\centering
		\begin{tabular}{||c c c c c||} 
			\hline
			Partition $(j)$ & $r$& $\tilde\prod_{i=1}^rz^{\beta_i}$ & $z^{\hat{\beta}_j}$& $\bn_i$ \\ [0.5ex] 
			\hline\hline
			1 & 1& $z_{w_1}$  & $z_{w_1}z_{w_2}z_{w_3}$ &$\bn_1 = \bn+\ell$ or $\bn_1 = \mathbf{m}+\ell$ \\[1ex] 
			\hline
			2 & 1& $z_{w_1}z_{w_2}$  & $z_{w_1}z_{w_3}$& $\bn_1 = \mathbf{m}+\ell$  \\ [1ex]
			\hline
			3 & 1& $z_{w_1}^2z_{w_3}$ & $z_{w_2}$& $\bn_1 = \bn+\ell$  \\ [1ex]
			\hline
			4 & 2& $z_{w_1} \odot z_{w_1}$   & $z_{w_2}z_{w_3}$ & $\bn_1 = \bn+\ell_1, \bn_2 = \mathbf{m}+\ell_2$ \\
%			&&&&or $\bn_1 = \mathbf{m}+\ell_1, \bn_2 = \bn+\ell_2$ \\
			&&&&or $\bn_1 = \mathbf{m}+\ell_1,
			\bn_2 = \mathbf{m}+\ell_2$  \\ [1ex]
			\hline
			5 & 2& $z_{w_1}\odot z_{w_1}z_{w_2}$  & $z_{w_3}$& $\bn_1 = \mathbf{m}+\ell_1,
			\bn_2 = \mathbf{m}+\ell_2$ \\ [1ex] 
			\hline
		\end{tabular}
	\end{table}
In the partition 4, we have ruled out the term $\bn_1 = \mathbf{m}+\ell_1, \bn_2 = \bn+\ell_2$ as one can check that both $\bn_1 = \mathbf{m}+\ell_1, \bn_2 = \bn+\ell_2$ and $\bn_1 = \bn+\ell_1, \bn_2 = \mathbf{m}+\ell_2$ give the same forest $z_{w_1}D^{(\bn+\ell_1 )} \odot z_{w_1}D^{(\mathbf{m}+\ell_2 )}$.
\item Then, we compute all possible $k$ such that 
$\bar{\partial}^k   z^{\hat{\beta}} \ne 0$. 
For example, in partition 1, $z^{\hat\beta_1}z_{w_1}z_{w_1}z_{w_1}$ has $b_0b_1$ 2 letters in $u$. Thus, there are $2^2 = 4$ possible $k$: $[0,0]$, $[1,0]$, $[0,1]$, $[1,1]$.
\begin{table}[H]
	\centering
	\begin{tabular}{||c c c c c||} 
		\hline
		Partition $(j)$ & $r$& $\tilde\prod_{i=1}^rz^{\beta_i}$ & $z^{\hat{\beta}_j}$& possible $k$ \\ [0.5ex] 
		\hline\hline
		1 & 1& $z_{w_1}$  & $z_{w_1}z_{w_2}z_{w_3}$ &$[0,0]$, $[1,0]$, $[0,1]$, $[1,1]$ \\[1ex] 
		\hline
		2 & 1& $z_{w_1}z_{w_2}$  & $z_{w_1}z_{w_3}$& $[0,0]$, $[1,0]$, $[0,1]$, $[1,1]$ \\ [1ex]
		\hline
		3 & 1& $z_{w_1}^2z_{w_3}$ & $z_{w_2}$& $[0,0]$  \\ [1ex]
		\hline
		4 & 2& $z_{w_1} \odot z_{w_1}$   & $z_{w_2}z_{w_3}$ & $[0,0]$, $[0,1]$ \\[1ex]
		\hline
		5 & 2& $z_{w_1}\odot z_{w_1}z_{w_2}$  & $z_{w_3}$& $[0,0]$, $[0,1]$ \\ [1ex] 
		\hline
	\end{tabular}
\end{table}
\item Finally, we combine possible $\bn_i$ and $k$ and apply formulae in Proposition~\ref{adjoint_SPDEs}. Here we only take $\bar{D}^{(\bn+\ell)} \bar{\partial}^{[0,1]}   z^{\hat{\beta}_1}$ as an example, as the calculations for other terms are the same. Since only $w_3$ has $b_1$,
\begin{equs}
	\bar{\partial}^{[0,1]}   z^{\hat{\beta}_1} = 1\frac{0+1}{1} z_{w_1}z_{w_2}z_{\mathbf{m}\mathbf{m}} = z_{b_0}z_{\bn}z_{\mathbf{m}\mathbf{m}}.
\end{equs}
Then we have
\begin{equs}
	\bar{D}^{(\bn+\ell)} z_{b_0}z_{\bn}z_{\mathbf{m}\mathbf{m}} &= 
	\sum_{\ell \in \mathbb{N}^{d+1}}
	\frac{(0+1)(\hat\beta(\ell)+1)}{\ell!1} z_{\ell}z_{b_0}z_{\mathbf{m}\mathbf{m}} 
	\\&= 2z_{b_0}^2z_{\mathbf{m}\mathbf{m}} + \sum_{\ell \ne [1,0]}\frac{1}{\ell!} z_{\ell}z_{b_0}z_{\mathbf{m}\mathbf{m}}.
\end{equs}
On the other hand, the forest is 
\begin{equs}
	\partial^{[0,1]} z_{w_1} D^{(\bn+\ell)}.
\end{equs}
Therefore the corresponding term in the coproduct is 
\begin{equs}
	2\partial^{[0,1]}z_{b_0}D^{(\bn+[1,0])} \otimes z_{b_0}^2z_{\mathbf{m}\mathbf{m}} + \sum_{\ell \ne [1,0]}\frac{1}{\ell!} \partial^{[0,1]} z_{b_0} D^{(\bn+\ell)} \otimes z_{\ell}z_{b_0}z_{\mathbf{m}\mathbf{m}}.
\end{equs}
\end{itemize}
\end{example}

We then introduce a new type of SPDE multi-indices $z^\alpha \in \CM_0$
\begin{equation*}
	z^\alpha := \prod_{w \in \cA_c} z_{(0,w)}.
\end{equation*}
The dual of the extraction-contraction corpoduct is the product $\star_1$ called simultaneous insertion  which can be derived from applying the Guin-Oudom process to the product $ \blacktriangleright $. 
For $z^\beta\in \CM_{\mathcal{R}_c}$ and $z^\alpha \in \CM_0$, $ \blacktriangleright $ is defined to be
\begin{equs} \label{insertion_SPDE}
	z^\beta \blacktriangleright z^\alpha : =\sum_{w\in \cA_c}\left(D^{w}z^{\beta}\right)
	\left(\partial_{z_{(0,w)}}z^\alpha\right),
\end{equs}	
where, for $w = uv \in \cA_c$, we define $D^{w} := \prod_{i=0}^d \partial_i^{u[b_i]}\prod_{j=1}^{|v|}D^{(v_j)}$. 
Furthermore, the product simultaneously inserting $z^{\beta_i} \in \CM_{\mathcal{R}_c}$ into $z^\alpha \in \CM_{0}$ is
\begin{equs} \label{insertion_multiple_SPDE}
	\tilde{\prod}_{i=1}^n z^{\beta_i} \star_1 z^\alpha:= \sum_{(w_i)_{i=1,...,n}}  
\prod_{i=1}^n	\left( D^{w_i} z^{\beta_i}\right)\left[\left(\prod_{i=1}^n\partial_{z_{(0,w_i)}}\right)z^\alpha\right].
\end{equs}
Similar to the ODE case, we shall require that $n = |z^\alpha|$.
As for $\Delta$, when we look at the adjoint map $\Delta^{\!-}$, one has infinite sum due to the operators $D^{w_i}$.
The same bigrading is used for making sense of these sums.
 \begin{proposition} \label{extraction_SPDE_1}
	The explicit formula of the extraction-contraction multi-indices coproduct $\Delta^{\!-} $ in the context of SPDE is 
		\begin{equs} \label{explicit_formula_extraction_1_SPDE}
			\begin{aligned}
		\Delta^{\! -} z^{\beta}  =
		\sum_{\tilde{\prod}_{i=1}^n z^{\beta_i} \in \CF}
		\sum_{\substack{w_1,...,w_n \in \cA_c\\ \prod_{i} z_{(0,w_i)} \in \CM_0}}   	
		E(\tilde{\prod}_{i=1}^nz^{\beta_i}, z^\alpha ,z^{\beta})  
		\tilde{\prod}_{i=1}^n z^{\beta_i}
		\otimes   z^\alpha, 
	\end{aligned}
	\end{equs}
	with
	\begin{equs}
	E(\tilde{\prod}_{i=1}^nz^{\beta_i}, z^\alpha ,z^{\beta}) 
		=  	\sum_{\beta  = \hat{\beta}_1 + \cdots + \hat{\beta}_n} 
		\frac{\alpha!S(z^\beta)}{S(z^\alpha)S( \tilde{\prod}_{i=1}^n z^{\beta_i})}
		\prod_{i=1}^n
		\frac{\langle  
			D^{w_i}z^{\beta_i}, z^{\hat{\beta}_i} \rangle }{S(z^{\hat{\beta}_i})},
	\end{equs}
where $z^\alpha := \prod_{i=1}^n z_{(0,w_i)}$ and we have order among $\hat{\beta}_1, \cdots , \hat{\beta}_n$.
Note that $ n \ge 1$ as we cannot insert the empty multi-indice.	
\end{proposition}
\begin{proof} The proof is similar to the one of Proposition~\ref{LOT_coproduct2} with the $k_i$ replaced by the $w_i$.
\end{proof}
By using the adjoint maps given in Proposition~\ref{adjoint_SPDEs} we can have the following alternative formula. 
\begin{theorem} \label{extraction_SPDE_A}
	An explicit formula of the extraction-contraction multi-indices coproduct $\Delta^{\!-} $ in the context of SPDE is 
	\begin{equs} \label{explicit_formula_extraction_1_SPDE_2}
		\begin{aligned}
			\Delta^{\! -} z^{\beta}  =
		\sum_{\tilde{\prod}_{i=1}^n z^{\beta_i} \in \CF}
		\sum_{\substack{w_1,...,w_n \in \cA_c\\ \prod_{i} z_{(0,w_i)} \in \CM_0}}   	
		E(\tilde{\prod}_{i=1}^nz^{\beta_i}, z^\alpha ,z^{\beta})  
		\tilde{\prod}_{i=1}^n z^{\beta_i}
		\otimes   z^\alpha,
	 \end{aligned}
	\end{equs}
	with
	\begin{equs}
		E(\tilde{\prod}_{i=1}^nz^{\beta_i}, \prod_{i=1}^n z_{(0,w_i)} ,z^{\beta})   
		=  	\sum_{\beta  = \hat{\beta}_1 + \cdots + \hat{\beta}_n}	  \frac{\alpha!}{S(z^\alpha)S_{\tiny{\text{ext}}}(\tilde{\prod}_{i=1}^n z^{\beta_i})}
		\prod_{i=1}^n
		\frac{\langle  
			z^{\beta_i}, \bar D^{w_i} z^{\hat{\beta}_i} \rangle }{S(z^{{\beta}_i})}
	\end{equs}
	where $\bar D^{w_i}$ is the short-hand notation of $\prod_{j=1}^r\bar D^{(\bn_j)}  \bar\partial^{k}$ for $w_i = uv$ where $u=k$ and 
	$v= \bn_1,\cdots,\bn_r$. Moreover, we have order among $\hat{\beta}_1, \cdots , \hat{\beta}_n$.
	Note that $ n \ge 1$ as we cannot insert the empty multi-indice.	
\end{theorem}
\begin{proof}
	The proof is similar to the one of Theorem~\ref{LOT_coproduct_3} with the $k_i$ replaced by the $w_i$.
	\end{proof}

\end{document}